\documentclass[11pt]{article}
\usepackage{epsfig}
\usepackage{amssymb,amsmath,amsthm,amscd}
\usepackage{latexsym}
 \usepackage{bbm,dsfont}
\usepackage{mathrsfs}

\newtheorem{thm}{Theorem}[section]
\newtheorem{prop}[thm]{Proposition}

\newtheorem{lem}[thm]{Lemma}

\theoremstyle{definition}
\newtheorem{defn}[thm]{Definition}
\newtheorem{rem}[thm]{Remark}

\newcommand{\be}{\begin{equation}}
\newcommand{\ee}{\end{equation}}

\newcommand{\R}{\mathbb{R}}
\newcommand{\N}{\mathbb{N}}
\newcommand{\E}{\mathbb{E}}

\newcommand{\p}{\mathbb{P}}

\def \eps {{ \varepsilon }}

\def \o {{  {\mathcal{O}} }}

\def \calf {{  {\mathcal{F}} }}

\def \calb {{  {\mathcal{B}}  }}

\def \caly {{  {\mathcal{Y}}  }}

 \def \call {{  {\mathcal{L}}  }}
 \def \calt {{  {\mathcal{T}}  }}

\begin{document}

\baselineskip=1.1 \baselineskip

\begin{center}
 {  \bf 
 Martingale Solutions of Fractional
 Stochastic 
 Reaction-Diffusion Equations
 Driven by  Superlinear Noise 
   }
\end{center}

\medskip

\medskip

\begin{center}
 Bixiang Wang  
\vspace{1mm}\\
Department of Mathematics, New Mexico Institute of Mining and
Technology \vspace{1mm}\\ Socorro,  NM~87801, USA \vspace{1mm}\\
Email: bwang@nmt.edu\vspace{2mm}\\
\end{center}
 



 \vspace{3mm}

\begin{abstract}  
  In this paper, we prove the
  existence of martingale solutions
  of a class of stochastic equations
  with   pseudo-monotone drift of polynomial
  growth of arbitrary order
  and a continuous diffusion term
  with
  superlinear growth.
  Both the nonlinear drift and diffusion terms
  are not required to be locally Lipschitz
  continuous.  We then  apply the 
  abstract result to establish the existence
  of martingale solutions of the
  fractional stochastic reaction-diffusion
  equation 
  with polynomial drift  driven by a superlinear noise.
  The pseudo-monotonicity techniques
  and the Skorokhod-Jakubowski
      representation theorem in 
      a topological space
      are used to pass to the limit of a sequence
      of approximate solutions
      defined by the Galerkin method.
   \end{abstract}



\section{Introduction}
\setcounter{equation}{0}

In this paper, we investigate the existence of martingale
solutions of the
stochastic  fractional reaction-diffusion equation
in a bounded domain
with a  monotone drift driven by a superlinear
noise.  Both the nonlinear drift and  diffusion terms
are continuous but not necessarily locally
Lipschitz continuous.
 
Let $\o$ be a bounded domain in $\R^n$.
For every $s\in (0,1)$,
consider the  non-local It\^{o}  
stochastic equation defined in $\o$
for $t>0$:
\be\label{rde1}
  du (t)
  + (-\Delta)^s  u (t)  dt
  =  f(t,x, u(t)) dt
  +   h(t,x, u(t) )
  dt
  +\sum_{i=1}^\infty \sigma_i  (t, x,  u(t))    {dW_i} ,
  \ \ x\in \o,
  \ee 
  with  boundary condition
  \be\label{rde2} 
  u(t,x) = 0, \quad x\in \R^n\setminus \o, \ t>0,
  \ee
  and initial condition
 \be\label{rde3}
 u( 0, x ) = u_0  (x),   \quad x\in  \R^n ,
 \ee 
  where 
 $(-\Delta)^s$ is  the fractional Laplace operator,
 $f: \R \times \R^n \times \R
 \to \R$ is
 continuous  and 
  decreasing 
  in its third argument,
  $h: \R \times \R^n \times \R
 \to \R$ is a Lipschitz function
 in its third argument,
 $\sigma_i: 
  \R \times \R^n \times \R
 \to \R$ is superlinear and continuous in its
 third argument,
 and $\{W_i\}_{i=1}^\infty$
 is a sequence of  independent
 real-valued  standard Brownian motion
    defined 
 on a  complete filtered probability space
 $(\Omega, \mathcal{F}, \{ { \mathcal{F}} _t\} _{t\in \R},  \p  )$.

  The 
     fractional partial differential equations 
     have numerous applications in probability and 
   physics
  \cite{abe1,  garr1,  guan2, 
    jara1, kos1}.
    In the deterministic case,
    the  solutions of such equations
    have been studied in  
   \cite{abe1,  caff1,  dine1,  gal1, garr1, guan2, 
    jara1, kos1,  luh1,   ros1,  ser2}.
    In the stochastic case, the existence of solutions
    and their long term dynamics have been investigated
    in  
     \cite{gu1, luh2,  wan1, wan2} and the references therein.
     In particular, the existence of solutions
     to the stochastic  system
      \eqref{rde1}-\eqref{rde3}
     were established  in \cite{wan1, wan2}
     when  the nonlinear function $f(t,x, u)$ is differentiable
     or locally Lipschitz continuous in $u$,
     and $\sigma_i(t,x,u)$ is 
     locally Lipschitz continuous 
    in $u$ with linear growth.
    
    In the present paper, we study the existence of martingale
    solutions to \eqref{rde1}-\eqref{rde3}
    and improve the results of \cite{wan1, wan2}
    in the following aspects: 
    
    (i) $f(t,x, u)$ is decreasing and continuous in $u$
    with polynomial growth of arbitrary
    order, but
    not necessarily locally Lipschitz continuous in $u$.
    
     (ii) $\sigma_i(t,x,u)$ has a superlinear growth
    instead of a linear growth  in $u$.
    
    (iii) $\sigma_i(t,x,u)$ is  only continuous in $u$, but
    not necessarily locally Lipschitz continuous in $u$.

    Actually, we will prove an abstract
    result
    (see Theorems \ref{main} and   \ref{main1})
     on the existence of martingale
    solutions for  stochastic
    partial differential equations
    with  pseudo-monotone  drift of polynomial 
    type 
    of any order,  which is applicable
   to  a large class of stochastic systems
    including the fractional
    reaction-diffusion equation
    \eqref{rde1}-\eqref{rde3},
    the fractional  $p$-Laplace equation
    and the tamed  
    Navier-Stokes equation
    with a polynomial drift of arbitrary order.
    We remark that
    Theorem \ref{main} in the present
    paper is an extension of
    Theorem 2.6 in \cite{roc1} which
    applies to 
     \eqref{rde1}-\eqref{rde3}
     only when
     $\sigma_i(t,x,u)$ has a linear growth
      in $u$ and 
    $f(t,x, u)$  has a restricted growth
    in  $u$
    (that is, the growth of
     $f(t,x, u) $
    in  $u$
    cannot be arbitrarily large).
    
    We mention that the nonlinear diffusion term
    in Theorem \ref{main}  does not depend
    on  spatial derivatives of solutions,
     and
    has  a superlinear but subcritical growth.
    In a forthcoming paper, 
     we will investigate  the case
    where   the growth of noise is critical   and the 
    diffusion      depends on
      derivatives of solutions 
    (i.e., transport noise).

    We will prove the existence of
    martingale solutions for
    \eqref{rde1}-\eqref{rde3}
    when the initial  data
    $u_0$ belong
     to 
    the space
    $H
    =\{ u\in L^2(\R^n):
    u=0 \text{ a e. on } \R^n \setminus \o
    \}$.
    We also need the space
    $V_1 =\{ u\in H^s (\R^n):
    u=0 \text{ a e. on } \R^n \setminus \o
    \}$ for the fractional Laplace
    operator $(-\Delta)^s$,
    and 
     $V_2 =\{ u\in L^p  (\R^n):
    u=0 \text{ a e. on } \R^n \setminus \o
    \}$ for the polynomial nonlinearity
    $f$.
    If the growth of  $f$ is restricted
    such that $V_1$ is embedded into
    $V_2$, then we may only consider
    the spaces $H$ and $V_1$ instead
    of $H, V_1$ and $V_2$
  as  in \cite{roc1}.
  Since the nonlinear term $f$
  has a polynomial growth
  of arbitrary order in this paper,
  we must  use the space 
  $V_1$ as well as  $V_2$.
  That is the reason why we deal with
  a finite number of 
  Banach spaces
  in Theorem \ref{main}
    to accommodate
    a variety of   nonlinearity
    in  a  given  stochastic differential equation.
    
    The proof of Theorem \ref{main}
    follows the general pseudo-monotone
     argument
    of \cite{roc1}, but  with necessary modifications.
    In particular,  we need to consider
    the tightness of a sequence of approximate
    solutions in the space
    $L^\infty(0,T; H)$ with
    weak-* topology
    as well as the space
    $L^{q_j}(0,T; V_j)$
    with weak topology
    for some $q_j> 1$
    with $j=1,\cdots J$.
    Since these spaces are not metric spaces,
    we cannot apply the classical
    Skorokhod representation theorem in a metric
    space as in \cite{roc1}.
    Instead,   we must employ the
     Skorokhod-Jakubowski
      representation theorem in 
      a topological space
      like 
       $L^\infty(0,T; H)$ with
    weak-* topology
    and  $L^{q_j}(0,T; V_j)$
    with weak topology.
    
    We remark that
    the authors of \cite{roc1} used the
   {\it  strong}  Skorokhod  representation theorem in 
    a metric space to obtain the existence of martingale solutions.
    Unfortunately, as reported recently in \cite{ond1}
    (see also \cite{pec1}),
    the  strong Skorokhod  representation theorem
    is incorrect even in a Polish space.
    Therefore, 
  in the present paper, we cannot use the strong
  representation theorem, instead 
    we have to employ
    the standard 
    Skorokhod-Jakubowski
      representation theorem in 
      a topological space.
       In addition, the proof of this paper for the 
      tightness
      of  approximate solutions
      in $L^r(0,T; H)$ is different from
      that of \cite{roc1}.

    Based on Theorem \ref{main},
    we then prove the existence of
    martingale solutions
    of \eqref{rde1}-\eqref{rde3}
    when $f$ is monotonically
    decreasing with a polynomial
    growth of  arbitrary order $p-1$
    with $p>2$
    and $\sigma_i$ is
    a  continuous function
    with a polynomial growth
    of order less than ${\frac 12}p$
    (see \eqref{f1}-\eqref{f3} and
    \eqref{sig1}-\eqref{sig4}
    for
    the  precise assumptions
    on $f$ and $\sigma_i$).
    Under these assumptions,
    we  will  verify
    that  $f$ and $\sigma_i$
    satisfy all conditions of
    Theorem \ref{main}
    including the necessary continuity
    of the noise coefficients
    in appropriate spaces
    (see Theorem \ref{main_rde}).
    
    It is worth mentioning that
    the solutions of stochastic
    differential equations
    with monotone nonlinearity
    has been extensively studied in the
    literature, see, e,g,
    \cite{gol1, kry1, liu1, liu2, mar1, ngu1,
    par1,  roc1, sal1, sal2, 
    val1,
    wanr1,  
    zha1} and the references therein.

     In the next section, we prove 
     the existence of martingale solutions
     for a general stochastic equation
     with monotone nonlinearity,
     and in the last section, we  show
      the existence of martingale solutions
     to the fractional stochastic system
   \eqref{rde1}-\eqref{rde3}
     when  $f$ is a decreasing
     function with polynomial growth
     of any order 
     and $\sigma_i$ is a continuous function
     with superlinear growth.

\section{Martingale solutions of an abstract equation}
\setcounter{equation}{0}

           In this section, we prove the existence
           of martingale solutions
           for a general stochastic equation
           by the Galerkin method.
           We first discuss
           the  assumptions on
           the nonlinearity and present
           the main  result.
           We then
              derive the uniform
           estimates of a sequence of 
           approximate solutions
           and  establish the tightness
           of distributions of these solutions.
           We finally  pass to the limit of
           the approximate solutions
           by 
           the monotone argument
           as in \cite{roc1}.

\subsection{Assumptions and main results}

      Let $H$ be a separable Hilbert space with norm
      $\| \cdot \|_H$ and inner product $(\cdot, \cdot)_H$. 
      Let $J\in \N$ be fixed and for each $j=1,2\cdots J$,
      $V_j$ be a separable  reflexive Banach space
      with norm $\| \cdot \|_{V_j}$.
       Suppose
       $H$ and $V_j$ 
  for all $ 1 \le  j \le J$ are continuously embedded in a Hausdorff topological vector space  $\mathbb{X}$.
      Denote by
      $V=\bigcap_{j=1}^J V_j$ with norm
      $\|\cdot \|_V
      =\sum_{j=1}^J \|\cdot \|_{V_j}$. 
      Then $V$ is a separable Banach space.
      We further assume that $V$ is  reflexive.
      Let $H^*$, $V_j^*$ and $V^*$ be the dual spaces of
      $H$, $V_j$ and $V$
      with  
      duality pairings
      $(\cdot, \cdot)_{(H^*, H)}$,
      $(\cdot, \cdot)_{(V_j^*, V_j)}$,
      and
      $(\cdot, \cdot)_{(V^*, V)}$, respectively.
       
       We   assume that
       $V$ is densely embedded into $H$. Then 
       we have the Gelfand  triple:
       $$ V \subseteq H \equiv H^* \subseteq  V^*,
       $$
       where  $H^*$ is identified with $H$
       by the Riesz  representation thorem.
         If $f\in H$ and $x\in V$, then
       $$
       (f,x)_{(V^*, V)} 
       = (f,x)_{H}.
       $$
       
   If $f \in  V_j^*$, then
   the restriction of $f$ to $V$ belongs to $V^*$,
   and hence $V_j^*$ can be considered as a 
   subspace of $V^*$ for every 
        $j=1,2, \cdots,  J$.
       Let   $\sum_{j=1}^J V_j^*$ be the  space endowed 
       with norm
       $$
       \| f \|_{\sum_{j=1}^J V_j^*}
       =\inf \left  \{
       \sum_{j=1}^J \|f_j\|_{V_j^*},
       \  f=\sum_{j=1}^J f_j , \ f_j \in V_j^* \ \text{ for } \ j=1,2, \cdots,
       J
       \right \},
       \quad \forall \ f\in \sum_{j=1}^J V_j^*.
       $$
       Then
        $\sum_{j=1}^J V_j^*$ is embedded into   $V^*$.
        If $v\in V$ and $f=\sum_{j=1}^J f_j \in \sum_{j=1}^J V_j^*
        \subseteq V^*$ with $f_j\in V_j^*$, then
        $$
        (f, v)_{(V^*, V)}
        =\sum_{j=1}^J
        (f_j, v)_{(V_j^*, V_j)}.
        $$

Let $(\Omega, \calf, (\calf_t)_{t\in [0,T]}, \p)$ be
a complete filtered probability space
satisfying the usual condition.
Let $U$ be a separable Hilbert space
and $I$ is the identity operator on $U$.
Suppose $(W(t), t\in [0,T])$ is a cylindrical
$I$-Wiener process in $U$  defined on
$(\Omega, \calf, (\calf_t)_{t\in [0,T]}, \p)$, which means
that there exists another separable Hilbert space
$U_0$ such the embedding $U \subseteq U_0$ is
Hilbert-Schmidt and $W$ takes values in $U_0$.
Denote by $\call_2 (U, H)$ the Hilbert space
of all Hilbert-Schmidt operators from $U$ to $H$
with norm $\| \cdot \|_{\call_2(U,H)}$.

For each $j=1,2,\cdots, J$, let
$A_j: [0,T] \times   V_j
\to V_j^*$ be a 
$(\calb ([0,T]) \times  \calb (V_j) , \calb (V_j^*))$-measurable
mapping, and    
  $A: [0,T]   \times V
\to V^*$ be the  mapping given by:
$$
A(t,   v)
=\sum_{j=1}^J A_j (t,   v),
\quad \forall \ t\in [0,T],
\ v\in V.
$$

Let $B:
 [0,T] \times V
\to \call_2(U,H)$ be a 
 $(\calb ([0,T]) \times \calb(V) , \calb (
\call_2(U,H)))$-measurable 
 mapping.
 Consider the following assumptions:

{\bf (H1)} (Hemicontinuity)
  For every $t\in [0,T]$ 
   and
  $v_1, v_2, v_3 \in V$,
  the mapping $(A(t,  v_1 +\lambda v_2),
   v_3)_{(V^*, V)} $ is continuous with respect to $\lambda \in \R$.

{\bf (H2)} (Local monotonicity)  For every 
$t\in [0,T]$ 
   and
  $u, v  \in V$,
  $$
  2(A(t,   u)
  - A(t,  v),
  \ u-v)_{(V^*, V)}
  + \| B(t,   u) -B(t,  v)^2_{\call_2(U,H)}
  $$
  \be
  \label{h2a}
  \le
  \left (
  g (t ) +  \varphi (u) + \psi (v)
  \right ) \| u-v\|^2_H,
\ee
  where $g \in L^1([0,T] )$, 
  and $\varphi, \psi: V \to \R$ are measurable  functions such that
  for all $u\in V$,
\be\label{h2b}
   |\varphi (u)|  + | \psi (u)|
   \le
   \alpha_1
   \left (
   1  
   +\sum_{j=1}^J  \|u \|_{V_j}^{q_j}
   \right ) \left ( 1+ \|u \|_H^\alpha
   \right ),
\ee
  where $\alpha_1 \ge 0 $, $\alpha \ge 0$, 
  and $q_j \in (1, \infty)$   are constants
  for all   $j=1,2,\cdots, J.$

  We mention that
   {\bf (H2)} is  needed only for the
   pathwise uniqueness of
   martingale  solutions, but not
   for the existence of martingale  solutions.
   As we will  see later, 
    the  following
   generalized local monotone 
   condition   is sufficient for the
   existence of martingale solutions
   of   stochastic equations.
  
  {\bf (H2)}$^\prime$ (General local monotonicity)  For every 
  $R>0$, there exists $h_R\in L^1([0,T]
  )$ such that for all
$t\in [0,T]$ 
   and
  $u, v  \in V$ with $\|u\|_V \le R$
  and $\| v\|_V \le R$,
 \be\label{h2c}
  2(A(t , u)
  - A(t , v),
  \ u-v)_{(V^*, V)} 
  \le 
  h_R (t )     \| u-v\|^2_H.
\ee

  {\bf (H3)} (Coercivity)
  For every $t\in [0,T]$ 
   and
  $ v \in V$,
  \be\label{h3a}
  2 (A(t , v),
   v)_{(V^*, V)}  
   + \|B(t, v)\|^2_{\call_2(U,H)}
   \le
   - \alpha_2 \sum_{j=1}^J \| v \|^{q_j}_{V_j}
   +
   g  (t ) (1 +    \| v \|^2_H)  ,
   \ee
   where 
   $\alpha_2>0$ is a constant.

   {\bf (H4)} (Growth of drift  terms)
  For every $j=1,2, \cdots, J$,  $t\in [0,T]$ 
   and
  $ v \in V$,
  \be\label{h4a}
   \| A_j (t,   v) \|^{\frac {q_j}{q_j -1}}_{V_j^*}
   \le   
    \alpha_3 \|v\|^{q_j}_{V_j} (1 +  \| v \|^\beta _H )
   +
    g  (t ) (1+   \| v \|^\beta _H),
   \ee
   where 
   $\alpha_3$ and $\beta$ are  
   nonnegative  constants.

    {\bf (H5)} (Growth of diffusion term)
  For every    $t\in [0,T]$ 
   and
  $ v \in V$,
  \be\label{h5a}
   \| B(t,   v) \|^2_{\call_2(U,H)}
   \le       \alpha_4 \sum_{j=1}^J \| v \|^{\hat{q}_j}_{V_j}
   +
    g  (t) (1+     \| v \|^2 _H),
   \ee
   where
   $1\le \hat{q}_j <q_j$
   and  $\alpha_4 \ge 0$  are  constants.
   In addition, if $v, v_n \in V$ such that
   $\{v_n\}_{n=1}^\infty$ is bounded in $V$ and 
    $v_n \to v$ in $H$
   as $n\to \infty$, then for all $t\in [0,T]$,
   \be\label{h5b}
  \lim_{n\to \infty}
   B(t,   v_n) =  B(t,  v)
   \ \text{ in }  \ \call_2 (U,H).
\ee
Furthermore, we assume that
if $v\in L^\infty(0,T; H)
\bigcap  (
\bigcap\limits_{1\le j\le J} L^{q_j}
(0,T; V_j)
  )$
  and 
  $\{v_n\}_{n=1}^\infty
 $ is a bounded sequence in
 $ L^\infty(0,T; H)
\bigcap  (
\bigcap\limits_{1\le j\le J} L^{q_j}
(0,T; V_j)
  )$
  such that
  $v_n  \to v$ in $L^1(0,T; H)$,
  then
  \be\label{h5c}
   \lim_{n\to \infty}
   B(\cdot ,   v_n) =  B(\cdot,  v)
   \ \text{ in }  \  L^2(0,T; \call_2 (U,H)).
\ee

For convenience,  we  set
\be\label{qj}
 {\tilde{q}} =\max_{1\le j\le J} q_j \  \ \text{and} \ \ 
{\underline{q}} =\min_ {1\le j\le J} q_j . \  \ 
\ee

Note that condition \eqref{h5a} indicates that
$B(t,v)$ may have  superlinear growth 
in $v$ if $\hat{q}_j>2$ for   some $j=1,2,\cdots, J$.
In addition, condition \eqref{h5b}-\eqref{h5c}
require  that $B(t,v)$ is  only  continuous in $v$,
 but not necessarily  locally Lipschitz continuous in $v$.

By   {\bf (H1)}, {\bf (H2)}$^\prime$
and {\bf (H4)}  one may verify that  the 
operator   $A(t,  \cdot): V\to V^*$
is 
 demicontinuous  
 for every $t\in [0,T]$  in the sense that
if $v_n \to v$ in $V$, then 
$A(t,  v_n) \to A (t,  v)$
weakly in $V^*$
 (see, e.g,  \cite[ Lemma 2.1, p. 1252]{kry1}
 or    \cite[Remark 4.4.1]{liu1}).
  Furthermore,  
   by    Lemma 2.15 in \cite{roc1},
 we know that if 
   the embedding  $V\subseteq H$
   is compact, then  for every 
   $t\in [0,T]$,  $A(t, \cdot): V\to V^*$
   is pseudo-monotone 
   which is  stated below.

   \begin{lem}[\cite{roc1}] 
   \label{pmon1}
   If {\bf (H1)} and {\bf (H2)}$^\prime$ hold and
   the embedding  $V\subseteq H$
   is compact, then  for every 
   $t\in [0,T]$,  $A(t, \cdot): V\to V^*$
   is pseudo-monotone in the sense that 
   if $v_n \to v$ weakly in $V$ and
   $$
   \liminf_{n\to \infty}
  ( A(t, v_n), v_n -v)_{(V^*, V)} \ge 0,
  $$
  then  for any $u\in V$,
    $$
   \limsup_{n\to \infty}
  ( A(t, v_n), v_n - u )_{(V^*, V)}  
  \le
   ( A(t, v), v - u )_{(V^*, V)} .
  $$
   \end{lem}
   


Consider the stochastic equation:
\be\label{sde1}
dX(t)
= A(t, X(t)) dt
+B(t,  X(t)) d W(t), \quad t\in (0, T],
\ee
with initial condition:
\be\label{sde2}
X(0) =x \in H.
\ee

The solution of \eqref{sde1}-\eqref{sde2}
is understood in the following sense.

\begin{defn}\label{dsol}
A continuous  $H$-valued 
$\calf_t$-adapted stochastic process
$(X(t): t\in [0,T])$ is called a solution
of problem \eqref{sde1}-\eqref{sde2}
if
$$
X\in L^2(\Omega, L^2(0,T; H))
\bigcap L^{q_j} (\Omega, L^{q_j}(0,T; V_j)),
\quad \forall \  j=1,2, \cdots, J,
$$ 
and for all $t\in [0,T]$,
$$
X(t) = x
+ \int_0^t A(s, X(s)) ds
+\int_0^t
B(s, X(s)) dW(s) \quad \text{in} \ \ V^*,
$$
$\p$-almost surely.
 \end{defn}

The main result of the paper is stated below.
 
 \begin{thm}\label{main}
 Suppose {\bf (H1)}, {\bf (H2)}$^\prime$
and {\bf (H3)}-{\bf (H5)} are fulfilled.
If  the embedding $V \subseteq H$ is compact,
then for any   $p \ge 1$  
and 
  $x\in H$,  \eqref{sde1}-\eqref{sde2}
has at least one martingale solution such that
\be\label{main 1}
\E \left (
\sup_{t\in [0,T]}
\|X(t)\|_H^{2p}
\right )
+
\E 
\left (
\left (
\sum_{j=1}^J \int_0^T
\| X(s)\|_{V_j}^{q_j}
ds
\right )^{p}
\right )
\le M(1 +\|x\|_H^{2p}),
\ee
where $M=M(T, p)>0$ is a constant
depending only on $T$ and $p$.

In addition, if {\bf (H2)}  is fulfilled,  then the
pathwise uniqueness of solutions  holds and
thus  \eqref{sde1}-\eqref{sde2}
has a unique solution in the sense of
Definition \ref{dsol}.
 \end{thm}

\subsection{Approximate solutions and uniform
estimates}

In this section,  we consider a sequence
of approximate solutions
 to \eqref{sde1}-\eqref{sde2}
by the Galerkin method, and then derive
the uniform estimates of these solutions.

Let $\{u^0_k\}_{k=1}^\infty$ be  
an orthonormal
  basis of $U_0$.
  Given  $n\in \N$,  let $Q_n: U_0 \to \  \text{span} \{u_1^0, \cdots,
  u^0_n\}$ be the  orthogonal  projection. Then we have
   $$
  Q_n W(t) = \sum_{k=1}^n (W(t), u^0_k)_{U_0} u^0_k,
  $$
  where $(\cdot, \cdot)_{U_0}$ is the inner product of $U_0$.

 Since $V$ is dense in $H$,  there exists
  an orthonormal
  basis   $\{h_k\}_{k=1}^\infty$   of $H$
  such that 
  $h_k \in V$ for all $k\in \N$
  and the finite linear combinations of
   $\{h_k\}_{k=1}^\infty$  is dense in $V$.
Given  $n\in \N$,  let $P_n: H \to  H_n=\text{span} \{h_1, \cdots,
  h_n\}$ be the  orthogonal  projection, which can be extended
  from $H$ to $V^*$ by
  $$
  P_n v^*=\sum_{k=1}^n (v^*, h_k)_{(V^*, V)} h_k,
  \quad \forall \ v^*\in V^*.
  $$

Consider the $n$-dimensional stochastic differential equation
for $Z_n \in  H_n$:
\be\label{ode1}
Z_n (t)
=P_n  x
+ \int_0^t P_n A(s, Z_n (s)) ds
+ \int_0^t P_n B(s,Z_n (s))Q_n  dW(s).
\ee

Since 
 $A(t, \cdot) :  V\to
V^*$ is  demicontinuous
  for every $t\in [0,T]$,  we find  that
$P_n A(t,  \cdot): H_n
\to H_n$ is continuous.
  On the other hand, by \eqref{h5b}  we see that
for every $t\in [0, T]$, the operator
$P_n B(t, \cdot)Q_n :  H_n
\to \call_2 (U, H_n)$ is continuous.
Then by \eqref{h3a}-\eqref{h5a}
and  the  theory on the existence 
of  solutions for
  stochastic ordinary differential equations 
  (see e.g., \cite{hof1}), we infer that
 problem \eqref{sde1}-\eqref{sde2}
has a  martingale  solution $Z_n$ on $[0,T]$
defined on a new probability space,
which is still denoted by  $(\Omega, \calf, \p)$.

  Next, we derive the uniform estimates
  on the sequence $\{Z_n\}_{n=1}^\infty$
  of approximate solutions.
  
  \begin{lem}\label{ues1}
If  {\bf (H1)}, {\bf (H2)}$^\prime$
and {\bf (H3)}-{\bf (H5)} are fulfilled,
then 
for any   $p\ge 1$  and 
 $x\in H$, 
  there exists a constant $M_1= M_1(T, p)>0$ 
  such that for all $n\in \N$, 
  the solution $Z_n$ of \eqref{ode1} satisfies,
  $$
  \E \left (
  \sup_{t\in [0,T]}
  \|Z_n (t)\|^{2p}_H
  \right  )
  +
   \E \left ( \left (
  \int_0^T \sum_{j=1}^J  \| Z_n (t) \|_{V_j}^{q_j} dt
  \right )^{p} \right )
  + \E \left (  \int_0^{T }
  \|Z_n (s)\|_H^{2p-2}  \sum_{j=1}^J
  \|Z_n (s)\|_{V_j}^{q_j} ds
  \right )
  $$
  $$
  \le M_1  (1 + \|x\|^{2p}_H).
   $$
  \end{lem}

  \begin{proof}
  The proof is standard by using It\^{o}'s formula. 
  For the reader's convenience,
  we sketch the  main steps.
  
  By \eqref{ode1} we have,
  $\p$-almost surely, for all $t\in [0,T]$,
  $$
  \|Z_n (t) \|^2_H
  -\|P_n x\|^2_H
  =
  2\int_0^t (A(s, Z_n(s)), Z_n (s) )_{(V^*, V)} ds$$
\be\label{ues1 p1}
  +\int_0^t \|P_n B(s, Z_n (s))Q_n\|^2_{\call_2(U,H)^2} ds
  +2 \int_0^t Z_n^*(s) P_n B(s, Z_n(s)) Q_n dW(s),
\ee
  where $Z_n^*(s)$ is the element in $H^*$ identified
  with $Z_n(s)$
  by Riesz's representation theorem.
  By \eqref{ues1 p1} and It\^{o}'s formula again we get,
  $\p$-almost surely, for all $t\in [0,T]$,
  $$
  \|Z_n (t)\|_H^{2p}
  -\|P_n x\|_H^{2p}
  =
 2 p\int_0^t \|Z_n (s)\|_H^{2p-2}
 (A(s, Z_n (s)), Z_n(s))_{(V^*, V)} ds
   $$
   $$
   +  p\int_0^t \|Z_n (s)\|_H^{2p-2}
  \|P_n B(s, Z_n (s))Q_n\|^2_{\call_2(U,H)^2} ds
  $$
  $$
   +  2p\int_0^t \|Z_n (s)\|_H^{2p-2}Z_n^*(s) P_n B(s, Z_n(s)) Q_n dW(s)
   $$
\be\label{ues1 p2}
   + 2p(p-1) \int_0^t \|Z_n (s)\|_H^{2p-4}
   \|Z_n^*(s) P_n B(s, Z_n(s)) Q_n\|^2_{\call_2(U,\R)} ds.
\ee
By \eqref{h3a}, \eqref{h5a}  and Young's inequality, 
we obtain from
\eqref{ues1 p2} that, 
$\p$-almost surely, for all $t\in [0,T]$,
  $$
  \|Z_n (t)\|_H^{2p}
  + p 
   \alpha_2  
    \int_0^t
  \|Z_n (s)\|_H^{2p-2}  \sum_{j=1}^J
  \|Z_n (s)\|_{V_j}^{q_j} ds
  $$
  $$
 \le \|P_n x \|_H^{2p}
  + p \int_0^t \left (  g (s) 
   \|Z_n (s)\|_H^{2p-2}+ g(s)  \|Z_n (s)\|_H^{2p}
  \right )ds
  $$
  $$
  + 2p (p-1) \alpha_4   \int_0^t
  \|Z_n (s)\|_H^{2p-2}  \sum_{j=1}^J
  \|Z_n (s)\|_{V_j}^{\hat{q}_j} ds
  $$
  $$
  + 2p(p-1) \int_0^t  \left (
  g (s)
   \|Z_n (s)\|_H^{2p-2}+g (s)  \|Z_n (s)\|_H^{2p}
  \right )ds
   $$
$$
   +  2p\int_0^t \|Z_n (s)\|_H^{2p-2}Z_n^*(s) P_n B(s, Z_n(s)) Q_n dW(s)
 $$
  $$
 \le \|P_n x \|_H^{2p}
   +c_1  \int_0^t  
  |g (s)| (1+ 
   \|Z_n (s)\|_H^{2p}  )
  ds
  $$
  $$
  +  {\frac 12} p\alpha_2  \int_0^t
  \|Z_n (s)\|_H^{2p-2}  \sum_{j=1}^J
  \|Z_n (s)\|_{V_j}^{{q}_j} ds
  +c_1  \int_0^t
  \|Z_n (s)\|_H^{2p-2} ds
  $$
 \be\label{ues1 p3}
   +  2p\int_0^t \|Z_n (s)\|_H^{2p-2}Z_n^*(s) P_n B(s, Z_n(s)) Q_n dW(s),
  \ee
 where $c_1=c_1(p)>0$ is a constant.
 It follows from \eqref{ues1 p3} that
 $\p$-almost surely, for all $t\in [0,T]$,
  $$
  \|Z_n (t)\|_H^{2p}
  +{\frac 12}
   p 
   \alpha_2  
    \int_0^t
  \|Z_n (s)\|_H^{2p-2}  \sum_{j=1}^J
  \|Z_n (s)\|_{V_j}^{q_j} ds
  $$
  $$
 \le \|P_n x \|_H^{2p}
   +c_1  \int_0^t  
  |g (s)| (1+ 
   \|Z_n (s)\|_H^{2p}  )
  ds
  +c_1  \int_0^t
  \|Z_n (s)\|_H^{2p-2} ds
  $$
 \be\label{ues1 p3a}
   +  2p\int_0^t \|Z_n (s)\|_H^{2p-2}Z_n^*(s) P_n B(s, Z_n(s)) Q_n dW(s).
  \ee

  Given $n\in \N$ and $R>0$, denote by
$$
\tau_{n,R}
=\inf  \left \{
t\ge 0: \|Z_n(t)\|_H 
+
\int_0^t \sum_{j=1}^J \| Z_n(s)\|^{q_j}_{V_j} ds
 \ge R
\right \} \wedge T,
$$ with $ \inf \emptyset = +\infty$.
Then by
 \eqref{ues1 p3a}
 we get,  $\p$-almost surely, for all $t\in [0,T]$,
  $$
  \|Z_n (t\wedge \tau_{n,R})\|_H^{2p}
  +  {\frac 12}
   p 
   \alpha_2     \int_0^{t\wedge \tau_{n,R}}
  \|Z_n (s)\|_H^{2p-2}  \sum_{j=1}^J
  \|Z_n (s)\|_{V_j}^{q_j} ds
  $$
  $$
 \le \|P_n x\|_H^{2p}
   +c_1  \int_0^{t\wedge \tau_{n,R}}  |g (s)|
(1+ 
   \|Z_n (s)\|_H^{2p} )
  ds
  +  c_1  \int_0^{t\wedge \tau_{n,R}} 
   \|Z_n (s)\|_H^{2p-2}  
  ds
  $$
  \be\label{ues1 p4a}
   +  2p\int_0^ {t\wedge \tau_{n,R}}  \|Z_n (s)\|_H^{2p-2}Z_n^*(s) P_n B(s, Z_n(s)) Q_n dW(s) 
\ee
 $$
 \le \|x\|_H^{2p}
    +c_1  \int_0^{t }
     |g (s)| ( 1+
   \|Z_n (s\wedge \tau_{n,R} )\|_H^{2p} 
   )
  ds
   +c_1  \int_0^{t }
   \|Z_n (s\wedge \tau_{n,R} )\|_H^{2p-2} 
   ds
  $$
 \be\label{ues1 p4}
   +  2p\int_0^ {t\wedge \tau_{n,R}}  \|Z_n (s)\|_H^{2p-2}Z_n^*(s) P_n B(s, Z_n(s)) Q_n dW(s).
  \ee

  Take the expectation of \eqref{ues1 p4} to obtain
  for all $t\in [0,T]$,
    $$
 \E \left ( \|Z_n (t\wedge \tau_{n,R})\|_H^{2p}
 \right )
  + {\frac 12} p\alpha_2
  \E \left (  \int_0^{t\wedge \tau_{n,R}}
  \|Z_n (s)\|_H^{2p-2}  \sum_{j=1}^J
  \|Z_n (s)\|_{V_j}^{q_j} ds
  \right )
  $$
  \be\label{ues1 p5}
 \le  \|x \|_H^{2p}  
  +c_1  \int_0^{T }
    |g (s)| ds
    +c_1  \int_0^{t }
   \left ( 1  + |g (s)| 
   \right )
    \E \left (
   \|Z_n (s\wedge \tau_{n,R} )\|_H^{2p} 
   \right ) 
  ds
  +{\frac 1p} c_1 T.
  \ee
By
  Gronwall's inequality, we get from \eqref{ues1 p5} that
 for all $t\in [0,T]$,
  $$
 \E \left ( \|Z_n (t\wedge \tau_{n,R})\|_H^{2p}
 \right )
  + {\frac 12}p \alpha_2   
  \E \left (  \int_0^{t\wedge \tau_{n,R}}
  \|Z_n (s)\|_H^{2p-2}  \sum_{j=1}^J
  \|Z_n (s)\|_{V_j}^{q_j} ds
  \right )
  $$
\be\label{ues1 p6}
 \le
 \left (   \| x \|_H^{2p}  
 + {\frac 1p} c_1 T
  + c_1  
  \int_0^ {T }    |g   (s )  |  ds
   \right )e^{  c_1  \int_0^{t } 
  (1+  |g (s)|)
    ds }.
  \ee
 Letting $R\to \infty$, by Fatou's lemma, we obtain
 for all $t\in [0,T]$,
  $$
 \E \left ( \|Z_n (t )\|_H^{2p}
 \right )
  +   {\frac 12} p\alpha_2 
  \E \left (  \int_0^{t }
  \|Z_n (s)\|_H^{2p-2}  \sum_{j=1}^J
  \|Z_n (s)\|_{V_j}^{q_j} ds
  \right )
  $$
     \be\label{ues1 p7}
 \le
 \left (   \| x \|_H^{2p}  
 +{\frac 1p} c_1 T
  + c_1  
  \int_0^ {T }    |g   (s )  |  ds
   \right )e^{  c_1  \int_0^{T } 
  (1+  |g (s)|)
    ds}  
   \le c_2  \left (1+
    \| x \|_H^{2p} \right )   ,
  \ee
 where $c_2 =c_2(T,p)>0$ is a constant.
 
   By 
 \eqref{ues1 p4a}
 we have 
  $$
  \E \left (
  \sup_{t\in [0,T]}\|Z_n (t\wedge \tau_{n,R})\|_H^{2p}
  \right )
 \le 
  \|x\|_H^{2p} 
  + c_1  \int_0^ {T } 
   |g  (s) | ds
   $$
   $$
  +c_1  \int_0^{T } 
 |g (s)|  \E \left (
   \|Z_n (s  )\|_H^{2p} \right )
  ds
   +c_1  \int_0^{T } 
  \E \left (
   \|Z_n (s  )\|_H^{2p-2}  \right )
  ds
  $$
 \be\label{ues1 p8}
   +  2p\E \left (
    \sup_{t\in [0,T]}
   \left |
   \int_0^ {t\wedge \tau_{n,R}}  \|Z_n (s)\|_H^{2p-2}Z_n^*(s) P_n B(s, Z_n(s)) Q_n dW(s)
   \right |\right ).
  \ee
  By \eqref{ues1 p7} we get
   $$
   c_1  \int_0^{T }
  |g (s)|  \E \left (
   \|Z_n (s  )\|_H^{2p} \right )
  ds 
  \le c_1  \int_0^{T }  
   |g (s)|  ds \sup_{s\in [0,T]} \E \left (
   \|Z_n (s  )\|_H^{2p} \right )
  $$
 \be\label{ues1 p9}
  \le c_1 c_2  \left (1+
   \| x \|_H^{2p} \right ) 
    \int_0^{T } 
 |g (s)|
    ds.
\ee
For the last term on the right-hand side
of \eqref{ues1 p8}, by the BDG inequality and \eqref{h5a}, we see
that there exists a constant $c_3=c_3(p)>0$ such 
that
 $$
 2p\E \left (
    \sup_{t\in [0,T]}
   \left |
   \int_0^ {t\wedge \tau_{n,R}}  \|Z_n (s)\|_H^{2p-2}Z_n^*(s) P_n B(s, Z_n(s)) Q_n dW(s)
   \right |\right )
   $$
   $$
\le c_3 \E \left ( \left (
   \int_0^ {T\wedge \tau_{n,R}}
     \|Z_n (s)\|_H^{4p-4}
     \| Z_n^*(s) P_n B(s, Z_n(s)) Q_n \|^2_{\call_2(U, \R)} ds
     \right )^{\frac 12} \right )
   $$
    $$
\le c_3 \E \left ( \left (
   \int_0^ {T\wedge \tau_{n,R}}
     \|Z_n (s)\|_H^{4p-2}
     \|     B(s, Z_n(s))   \|^2_{\call_2(U, H)} ds
     \right )^{\frac 12} \right )
   $$
    $$
\le c_3 \E \left ( \left (
   \int_0^ {T\wedge \tau_{n,R}}
    \|Z_n (s\wedge \tau_{n,R} )\|_H^{2p}
    \|Z_n (s)\|_H^{2p-2}
     \|     B(s, Z_n(s))   \|^2_{\call_2(U, H)} ds
     \right )^{\frac 12} \right )
   $$
    $$
\le c_3 \E \left ( 
 \left ( \sup_{s\in [0,T]} \|Z_n (s\wedge \tau_{n,R} )\|_H^{p}
   \right )
\left (
   \int_0^ {T\wedge \tau_{n,R}}
     \|Z_n (s)\|_H^{2p-2}
     \|     B(s, Z_n(s))   \|^2_{\call_2(U, H)} ds
     \right )^{\frac 12} \right )
   $$
    $$
\le{\frac 12}  \E 
 \left ( \sup_{s\in [0,T]} \|Z_n (s\wedge \tau_{n,R} )\|_H^{2p}
   \right ) 
   +
   {\frac 12} c_3^2
   \E
\left (
   \int_0^ {T\wedge \tau_{n,R}}
     \|Z_n (s)\|_H^{2p-2}
     \|     B(s, Z_n(s))   \|^2_{\call_2(U, H)} ds
     \right )
   $$
    $$
\le{\frac 12}  \E 
 \left ( \sup_{s\in [0,T]} \|Z_n (s\wedge \tau_{n,R} )\|_H^{2p}
   \right ) 
   +
   {\frac 12} c_3^2
   \E
\left (
   \int_0^ {T }
    \left ( g(s) \|Z_n (s)\|_H^{2p-2} 
    + g(s)
     \|Z_n (s)\|_H^{2p}  \right )
     ds
     \right )
   $$
    $$
    +
   {\frac 12} c_3^2 \alpha_4
   \E
\left (
   \int_0^ {T\wedge \tau_{n,R}}
     \|Z_n (s)\|_H^{2p-2}
     \sum_{j=1}^J
     \|  Z_n(s)    \|^{\hat{q}_j}_{V_j}
       ds
     \right )
   $$
     $$
\le{\frac 12}  \E 
 \left ( \sup_{s\in [0,T]} \|Z_n (s\wedge \tau_{n,R} )\|_H^{2p}
   \right ) 
 +c_4 
   \int_0^ {T }
      | g(s)| \left (1+
       \E \left (
     \|Z_n (s)\|_H^{2p}  \right )
     \right )
     ds 
   $$
    $$
    +
   {\frac 12} c_3^2 \alpha_4
   \E
\left (
   \int_0^ {T}
     \|Z_n (s)\|_H^{2p-2}
     \left (c_4 + 
     \sum_{j=1}^J
     \|  Z_n(s)    \|^{q_j}_{V_j}
     \right )
       ds
     \right )
   $$
    $$
\le{\frac 12}  \E 
 \left ( \sup_{s\in [0,T]} \|Z_n (s\wedge \tau_{n,R} )\|_H^{2p}
   \right ) 
   + 
c_4 
   \int_0^ {T }
    | g(s)|
    ds
    \left (1+ 
     \sup_{s\in [0,T]} \E \left (
     \|Z_n (s)\|_H^{2p}  \right )\right )
   $$
    $$
    +
   {\frac 12} c_3^2 \alpha_4
   \E
\left (
   \int_0^ {T}
     \|Z_n (s)\|_H^{2p-2} \left (c_4 +
     \sum_{j=1}^J
     \|  Z_n(s)    \|^{q_j}_{V_j}
     \right )
       ds
     \right )
   $$
 \be\label{ues1 p10}
\le{\frac 12}  \E 
 \left ( \sup_{s\in [0,T]} \|Z_n (s\wedge \tau_{n,R} )\|_H^{2p}
   \right ) 
   + 
  c_5 \left (1 + \| x \|^{2p}_H  
\right )   ,
\ee
 where  the last inequality follows from
 \eqref{ues1 p7}, and $c_5=c_5(T, p)>0$ is a constant.
   
   It follows from \eqref{ues1 p8}-\eqref{ues1 p10}
   that 
   $$   \E 
 \left ( \sup_{t\in [0,T]} \|Z_n (t\wedge \tau_{n,R} )\|_H^{2p}
   \right ) 
   \le c_6  \left (1 + \| x \|^{2p}_H 
\right ),
$$
 where $c_6=c_6(T, p)>0$ is a constant.
Letting $R\to \infty$, by Fatou's lemma we get
\be\label{ues1 p11}  \E 
 \left ( \sup_{t\in [0,T]} \|Z_n (t  )\|_H^{2p}
   \right ) 
   \le c_6  \left (1 + \| x \|^{2p}_H 
\right ).
\ee

On the other hand, by \eqref{ues1 p4a}
with $p=1$ we get, 
 $\p$-almost surely, for all $t\in [0,T]$,
  $$
  {\frac 12}  \alpha_2 \int_0^{t\wedge \tau_{n,R}}
  \sum_{j=1}^J
  \|Z_n (s)\|_{V_j}^{q_j} ds
  \le \| x\|_H^{2} +c_1T 
  +c_1  \int_0^{t\wedge \tau_{n,R}}  
   |g (s)| (1+ 
   \|Z_n (s)\|_H^{2}  )
  ds 
  $$
  $$
   +  2\int_0^ {t\wedge \tau_{n,R}}  
   Z_n^*(s) P_n B(s, Z_n(s)) Q_n dW(s),
    $$
  which implies that for any $p\ge 1$,
   $$
    2^{-p} \alpha_2^p
  \E \left (
  \left (
    \int_0^{T\wedge \tau_{n,R}}
  \sum_{j=1}^J
  \|Z_n (s)\|_{V_j}^{q_j} ds  \right )^p \right )
  \le 3^{p-1}  (
   \| x \|_H^{2}  +c_1T)^p
   $$
    $$
   +3^{p-1} c_1^p
   \E \left (
   \left (  \int_0^{T }  
 |g (s)| (1+ 
   \|Z_n (s)\|_H^{2} )
  ds 
  \right )^p \right )
  $$
  $$
   +   3^{p-1} 2^p
   \E \left (
   \left |
   \int_0^ {T\wedge \tau_{n,R}}  
   Z_n^*(s) P_n B(s, Z_n(s)) Q_n dW(s)
   \right |^p
   \right ),
    $$
    $$
  \le 3^{p-1} 
    (
   \| x \|_H^{2}  +c_1T)^p
  + 6^{p-1} c_1^p
    \left (  \int_0^{T}  
 |g (s)| ds  \right )^p
   \left (1 +   \E \left (\sup_{s\in [0,T]}
   \|Z_n (s)\|_H^{2p} \right )
   \right )
  $$
 \be\label{ues1 p12}
   +  3^{p-1}  2^{p}
   \E \left (
   \left |
   \int_0^ {T\wedge \tau_{n,R}}  
   Z_n^*(s) P_n B(s, Z_n(s)) Q_n dW(s)
   \right |^p
   \right ).
  \ee 
  For the last term in \eqref{ues1 p12},
  by the BDG inequality, \eqref{h5a},
  Young's inequality and \eqref{ues1 p11}  we have
$$
     3^{p-1}  2^{p}
   \E \left (
   \left |
   \int_0^ {T\wedge \tau_{n,R}}  
   Z_n^*(s) P_n B(s, Z_n(s)) Q_n dW(s)
   \right |^p
   \right )
$$
$$
  \le c_7  
   \E \left (
   \left  (
   \int_0^ {T\wedge \tau_{n,R}}  
   \| Z_n^*(s) P_n B(s, Z_n(s)) Q_n \|^2_{\call_2(U, \R)}
   ds
   \right )
   ^{\frac p2}
   \right )
   $$
   $$
  \le c_7  
   \E \left (
   \left  (
   \int_0^   {T\wedge \tau_{n,R}}     (
   \alpha_4    \| Z_n(s)  \|^2_H
   \sum_{j=1}^J
    \| Z_n(s)  \|_{V_j}^{\hat{q}_j}
   +  
   g(s) (
   \| Z_n(s)  \|^2_H
  +
   \| Z_n(s)  \|^4_H
    )   )
   ds
   \right )
   ^{\frac p2}
   \right )
   $$
    $$
  \le  2^{\frac p2}  \alpha_4 ^{\frac p2}
  c_7  
   \E \left (
   \left  (
   \int_0^  {T\wedge \tau_{n,R}}    
     \| Z_n(s)  \|^2_H
   \sum_{j=1}^J
    \| Z_n(s)  \|_{V_j}^{\hat{q}_j} ds
     \right )
   ^{\frac p2}
   \right )
   $$
   $$
   +     2^{\frac p2}
  c_7  
   \E \left (
   \left  (
   \int_0^ {T } 
  | g(s)|  (
   \| Z_n(s)  \|^2_H
  +
   \| Z_n(s)  \|^4_H
    )   
   ds
   \right )
   ^{\frac p2}
   \right )
   $$
    $$
  \le  2^{\frac p2}  \alpha_4 ^{\frac p2}
  c_7  
   \E \left (
   \sup_{s\in [0,T]}
   \| Z_n(s)  \|^p _H
   \left  (
   \int_0^{T\wedge \tau_{n,R}}  
   \sum_{j=1}^J
    \| Z_n(s)  \|_{V_j}^{\hat{q}_j} ds
     \right )
   ^{\frac p2}
   \right )
   $$
   $$
   +     2^{\frac p2}
  c_7   \left (
   \int_0^ {T } 
   |g(s)|  ds \right )^{\frac p2}
   \E \left (
   \left  (
   \sup_{s\in [0,T]}
   \| Z_n(s)  \|^2_H
  + \sup_{s\in [0,T]}
   \| Z_n(s)  \|^4_H
   \right )
   ^{\frac p2}
   \right )
   $$
     $$
  \le 
 2^{-p-1} \alpha_2^p 
    \E \left ( 
   \left  (
   \int_0^   {T\wedge \tau_{n,R}}    
   \sum_{j=1}^J
    \| Z_n(s)  \|_{V_j}^{{q}_j} ds
     \right )
   ^p
   \right )
  + c_8 + 
  c_8
   \E \left (
   \sup_{s\in [0,T]}
   \| Z_n(s)  \|^{2p} _H
   \right )
   $$
   $$
   +     2^{p}
  c_7   \left (
   \int_0^ {T } 
   |g(s)|  ds \right )^{\frac p2}
   \E \left (
   \left  (
   \sup_{s\in [0,T]}
   \| Z_n(s)  \|^p_H
  + \sup_{s\in [0,T]}
   \| Z_n(s)  \|^{2p} _H
   \right ) 
   \right )
   $$
   \be\label{ues1 p13}
    \le 
 2^{-p-1} \alpha_2^p 
    \E \left ( 
   \left  (
   \int_0^  {T\wedge \tau_{n,R}}    
   \sum_{j=1}^J
    \| Z_n(s)  \|_{V_j}^{q_j} ds
     \right )
   ^p
   \right )
 + c_9  \left (1 +  \| x \|^{2p}_H  
\right ),
\ee
    where $c_9=c_9(T, p)>0$ is a constant.
   It follows from \eqref{ues1 p11}-\eqref{ues1 p13}
   that
  \be\label{ues1 p14} 
   \E \left (
  \left (
    \int_0^{T\wedge \tau_{n,R}}
  \sum_{j=1}^J
  \|Z_n (s)\|_{V_j}^{q_j} ds  \right )^p \right )
    \le 
   c_{10}  \left (1 +  \| x\|^{2p}_H 
\right ),
\ee
   where $c_{10}=c_{10}(T, p)>0$ is a constant.
   Letting $R \to \infty$ in \eqref{ues1 p14},
    by Fatou's lemma we obtain 
   \be\label{ues1 p15}
  \E \left (
  \left (
    \int_0^{T }
  \sum_{j=1}^J
  \|Z_n (s)\|_{V_j}^{q_j} ds  \right )^p \right )
  \le 
   c_{10}  \left (1 + \| x \|^{2p}_H 
\right ).
\ee
  Then the desired estimates
  follow from \eqref{ues1 p7},
  \eqref{ues1 p11}
  and
  \eqref{ues1 p15}. 
    \end{proof}

   \begin{lem}\label{ues2}
   If  {\bf (H1)}, {\bf (H2)}$^\prime$
and {\bf (H3)}-{\bf (H5)} are fulfilled,
then 
for any     $x\in H$, 
  there exists a constant $M_2= M_2(T)>0$ 
  such that
  the solution $Z_n$ of \eqref{ode1} satisfies,
  for all $n\in \N$  and
   $j=1,\cdots, J$,
    $$ 
     \E \left (
     \int_0^T
   \| A_j (s,   Z_n(s) ) \|^{\frac {q_j }{q_j  -1}}_{V_j^*}
   ds \right )
   +
    \E \left (
  \int_0^T
  \| B(s, Z_n (s))\|^{2}_{\call_2(U,H)} ds
  \right )
    \le   
 M_2 (1+  \| x \|^{\beta +2}_H ).
  $$  
   \end{lem}
  
  \begin{proof}
  By  \eqref{h4a} and Lemma \ref{ues1} we get
    for all $j=1,\cdots, J$,
     $$
     \E \left (
     \int_0^T
   \| A_j (s,   Z_n(s) ) \|^{\frac {q_j}{q_j -1}}_{V_j^*}
   ds \right )
    \le   
    \alpha_3
    \E \left (\int_0^T
      \| Z_n (s) \|^{q_j}_{V_j} ds
      \right )
     $$
$$
+ \alpha_3 
 \E \left (
  \sup_{s\in [0,T]}  \| Z_n (s)  \|^\beta _H 
 \int_0^T
      \| Z_n (s) \|^{q_j}_{V_j} ds
    \right )
    $$
    $$
   + \int_0^T |g(s)| ds
   \left ( 1+   \E \left ( 
   \sup_{s\in [0,T]}
    \| Z_n (s)  \|^\beta _H 
    \right ) \right )
 $$
 $$
  \le   
    \alpha_3
    \E \left (\int_0^T
      \| Z_n (s) \|^{q_j}_{V_j} ds
      \right )
 + \alpha_3 
 \left (  \E \left (
  \sup_{s\in [0,T]}  \| Z_n (s)  \|^{2\beta} _H 
  \right ) \right )^{\frac 12}
   \left (  \E \left ( 
 \int_0^T
      \| Z_n (s) \|^{q_j}_{V_j} ds
    \right )^2
    \right )^{\frac 12}
    $$
    $$
   + \int_0^T |g(s)| ds
   \left ( 1+   \E \left ( 
   \sup_{s\in [0,T]}
    \| Z_n (s)  \|^\beta _H 
    \right ) \right )
 $$
 \be\label{ues2 p1}
 \le c_1 (1+ \| x  \|^{\beta +2}_H ),
 \ee
 where $c_1=c_1 (T, \beta)>0$ is a constant.

  On the other hand, by \eqref{h5a}
  and Lemma \ref{ues1} we have
  $$
  \E \left (
  \int_0^T
  \| B(s, Z_n (s))\|^{2}_{\call_2(U,H)} ds
  \right )
   \le
     \alpha_4
    \E \left (\int_0^T \sum_{j=1}^J
      \| Z_n (s) \|^{\hat{q}_j}_{V_j} ds 
     \right )
  $$
$$
+ 
  \int_0^T |g (s )|  ds  
  \left (1+  \E
  \left ( \sup_{s\in [0,T]}   \|Z_n (s) \|^{2}_H 
  \right ) \right )
  $$
   \be \label{ues2 p2}
  \le   
  c_2 (1+\|x \|^{2}_H ),
  \ee
  for some $c_2=c_4 (T)>0$.
  Then \eqref{ues2 p1}-\eqref{ues2 p2} complete the proof.
  \end{proof}

\subsection{Tightness of approximate solutions}
In this section, we prove the tightness
of the sequence
$\{Z_n\}_{n=1}^\infty$  of approximate
solutions in several spaces.
We start with the tightness of
$\{Z_n\}_{n=1}^\infty$
in 
$C([0,T], V^*)$.

      \begin{lem}\label{tig}
  Suppose   {\bf (H1)}, {\bf (H2)}$^\prime$
and {\bf (H3)}-{\bf (H5)}  hold.
If the embedding $V\subseteq H$ is compact,
then  
    the sequence $\{Z_n \}_{n=1}^\infty$
    is tight in $C([0,T], V^*)$.
     \end{lem}
     
     \begin{proof}
     We follow the argument of \cite{roc1}.
     Since   the uniform topology
   on   $C([0,T], V^*)$
   is the same as 
     the topology
     on  $C([0,T], V^*)$ induced by the 
     Skorokhod topology
     on  $D([0,T], V^*)$, 
     we only need to show 
     $\{Z_n \}_{n=1}^\infty$
     is tight in  $D([0,T], V^*)$.

       By Lemma \ref{ues1},   there exists
    $c_1=c_1(T)>0$ such that
    \be\label{tig p1}
    \E \left (
    \sup_{t\in [0,T]} \|Z_n (t)\|_H^2
    \right ) 
    \le c_1,
    \ee
 By \eqref{tig p1}, 
   for every   $\eps>0$, there exists
    $R=R(\eps, T)>0$ such that
    for all $n\in \N$,
    $$
    \p
    \left (
     \sup_{t\in [0,T]} \|Z_n (t)\|_H
      >R
    \right ) <\eps,
    $$
    that is
    \be\label{tig p2}
    \p \left (
    Z_n (t) \in B_H(R), \ \ \forall \ t\in [0,T]
    \right ) > 1-\eps,
   \ee
    where $B_H(R) =\{ h\in H: \|h\|_H \le R\}$.
    Note that $B_H(R)$ is a compact subset of $V^*$
    by the compact embedding $H\subseteq V^*$.
    
    Since $V^{**}$ separates the points of $V^*$ and $V$ is
    reflexive, we see that $V$ separates the points of $V^*$.
    Since the finite linear combinations of 
    $\{h_k\}_{k=1}^\infty$ are dense in $V$, we
    see that $\bigcup_{n=1}^\infty
    P_n H$ is dense in $V$ and thus also separates the points
    of $V^*$. Then  by Theorem 3.1 in \cite{jak1}
    and  \eqref{tig p2},
    we only need to
    show that for every $h\in 
    \bigcup_{n=1}^\infty
    P_n H$,  the  sequence 
     $(Z_n (t), h)_{(V^*, V)}$ is tight
     in $D([0,T], \R)$, for which, 
     by     \eqref{tig p2} and the 
    Aldous theorem in \cite{ald1}, we only need to
    prove for any stopping time
    $0\le \theta_n \le T$ and any
    $0\le \delta_n \le 1$ with $\delta_n \to 0$,
    \be\label{tig p3}
     \lim_{n\to \infty}
      (Z_n (\theta_n +\delta_n)
      -Z_n (\theta_n) , \  h)_{(V^*, V)}
    =0,
    \quad \text{in probability },
    \ee
    with the convention $\theta_n + \delta_n =T$
    if $ \theta_n + \delta_n>T$.
    
    Given $k, n\in \N$, define a stopping time
    $\tau_n^k$ by
    $$
    \tau_n^k
    =\inf
    \left \{
    t\ge 0:
    \ \|Z_n (t)\|
    +\int_0^t \|Z_n (s)\|_{V}^{\underline{q}}ds
    \ge k
    \right \} \wedge T,
   $$
   where  $ \underline{q}$ is given by \eqref{qj}.
   It follows from Lemma \ref{ues1} that
       \be\label{tig p4}
   \lim_{k\to \infty}
   \sup_{n\in \N}
   \p
   \left (
   \tau_n^k <T
   \right ) =0.
  \ee
  
  Given    $h\in 
    \bigcup_{n=1}^\infty
    P_n H$,  there exists
    $m\in \N$ such that
    $h\in P_m H$,  and  thus $P_nh =h$
    for all $n\ge m$ and 
\be\label{tig p4a}
    \sup_{n\in \N}
    \|P_n h\|_V  \le c_2,
\ee
    for some $c_2>0$.
    For every  $\delta>0$  and $\eta>0$, we have
  $$
  \p
  \left (
  \left |
  \left ( Z_n(\theta_n +\delta)
  -Z_n (\theta_n), \  h
  \right )_{(V^*,V)}
  \right | >\eta
  \right )
  $$
  $$
  \le 
  \p
  \left (
  \left |
  \left ( Z_n(\theta_n +\delta)
  -Z_n (\theta_n), \  h
  \right )_{(V^*,V)}
  \right | >\eta, \ \tau_n^k \ge T
  \right )
  +
    \p
  \left (
   \tau_n^k < T
  \right )
  $$
      $$
=
  \p
  \left (
  \left |
  \left ( Z_n((\theta_n +\delta)\wedge \tau_n^k )
  -Z_n (\theta_n\wedge \tau_n^k), \  h
  \right )_{(V^*,V)}
  \right | >\eta 
  \right )
  +
    \p
  \left (
   \tau_n^k < T
  \right )
  $$ 
        \be\label{tig p5}
\le {\frac 1\eta}
\E 
  \left (
  \left |
  \left ( Z_n((\theta_n +\delta)\wedge \tau_n^k )
  -Z_n (\theta_n\wedge \tau_n^k), \  h
  \right )_{(V^*,V)}
  \right |  
  \right )
  +
    \p
  \left (
   \tau_n^k < T
  \right ).
 \ee
 
 By \eqref{ode1} we have
  $$ 
  \left ( Z_n((\theta_n +\delta)\wedge \tau_n^k )
  -Z_n (\theta_n\wedge \tau_n^k), \  h
  \right )_{(V^*,V)}
   $$
\be\label{tig p6}
   = \int_{\theta_n\wedge \tau_n^k}
   ^{(\theta_n +\delta)\wedge \tau_n^k}
   (A(s, Z_n(s)), \  P_n h)_{(V^*, V)} ds
   + \int_{\theta_n\wedge \tau_n^k}
   ^{(\theta_n +\delta)\wedge \tau_n^k}
   h^* P_nB(s, Z_n (s))Q_n dW,
 \ee
    where $h^*$ is the element in $H^*$ identified
    with $h$ by the  Riesz  representation theorem.
    By \eqref{tig p6} we get
     $$ 
  \E \left (
  \left |
  \left ( Z_n((\theta_n +\delta)\wedge \tau_n^k )
  -Z_n (\theta_n\wedge \tau_n^k), \  h
  \right )_{(V^*,V)}
  \right | \right )
   $$
 \be\label{tig p7}
 \le
 \E \left (
  \int_{\theta_n\wedge \tau_n^k}
   ^{(\theta_n +\delta)\wedge \tau_n^k}
  \|A(s, Z_n(s))\|_{V^*} \| P_n h\|_V  ds
  \right )
   +
   \E \left (
   \left |
    \int_{\theta_n\wedge \tau_n^k}
   ^{(\theta_n +\delta)\wedge \tau_n^k}
   h^* P_nB(s, Z_n (s))Q_n dW
   \right |
   \right ) .
\ee
By \eqref{tig p4a} and Lemma \ref{ues2} we get
 $$
 \E \left (
  \int_{\theta_n\wedge \tau_n^k}
   ^{(\theta_n +\delta)\wedge \tau_n^k}
  \|A(s, Z_n(s))\|_{V^*} \| P_n h\|_V  ds
  \right )
     \le
  c_2 \delta^{\frac 1{\tilde{q}}}
   \E \left ( \left (
  \int_{\theta_n\wedge \tau_n^k}
   ^{(\theta_n +\delta)\wedge \tau_n^k}
  \|A(s, Z_n(s))\|_{V^*}^{\frac {\tilde{q}}{\tilde{q} -1}}    ds
  \right )^  {\frac {\tilde{q} -1}{\tilde{q} }} 
  \right )
  $$
  \be\label{tig p8}
     \le
  c_2 \delta^{\frac 1{\tilde{q}}}
  \left (   \E \left (
  \int_0
   ^{T}
  \|A(s, Z_n(s))\|_{V^*}^{\frac {\tilde{q}}{\tilde{q} -1}}    ds
  \right )
  \right )^  {\frac {\tilde{q} -1}{\tilde{q} }} 
   \le
  c_3 \delta^{\frac 1{\tilde{q}}},
\ee
  where
  $ {\tilde{q}}$ is given by \eqref{qj} and
   $c_3=c_3 (T)>0$ is a constant.
  For the last term in \eqref{tig p6}, by \eqref{h5a},
  Lemma \ref{ues1}
  and 
  the Burkholder inequality  we have
  $$ \E \left (
   \left |
    \int_{\theta_n\wedge \tau_n^k}
   ^{(\theta_n +\delta)\wedge \tau_n^k}
   h^* P_nB(s, Z_n (s))Q_n dW
   \right |
   \right )
   $$
   $$\le 3   \E \left (
   \left (
    \int_{\theta_n\wedge \tau_n^k}
   ^{(\theta_n +\delta)\wedge \tau_n^k}
   \| h^* P_nB(s, Z_n (s))Q_n \|^2_{\call_2(U,H)}
   ds \right )^{\frac 12}
   \right )
   $$
   $$\le 3 \|h\|   \E \left (
   \left (
    \int_{\theta_n\wedge \tau_n^k}
   ^{(\theta_n +\delta)\wedge \tau_n^k}
   \|  B(s, Z_n (s))  \|^2_{\call_2(U,H)}
   ds \right )^{\frac 12}
   \right )
   $$
    $$\le 3 \|h\|   \E \left (
   \left (
    \int_{\theta_n\wedge \tau_n^k}
   ^{(\theta_n +\delta)\wedge \tau_n^k}
   \left (
   \alpha_4 \sum_{j=1}^J
   \|  Z_n (s) \|_{V_j}^{\hat{q}_j}
   +
    g(s)(1+  \|  Z_n (s)   \|^2_H)
    \right )
   ds \right )^{\frac 12}
   \right )
   $$
     $$\le 3 \|h\|   \alpha_4 ^{\frac 12}
     \sum_{j=1}^J
       \E \left (
   \left (
    \int_{\theta_n\wedge \tau_n^k}
   ^{(\theta_n +\delta)\wedge \tau_n^k}
   \|  Z_n (s) \|_{V_j}^{\hat{q}_j}
  ds \right )^{\frac 12} \right )
   $$
     $$
     + 3 \|h\|   \E \left (
   \left (
    \int_{\theta_n\wedge \tau_n^k}
   ^{(\theta_n +\delta)\wedge \tau_n^k}
    |g(s)| (1+  \|  Z_n (s)   \|^2_H)
   ds \right )^{\frac 12}
   \right )
   $$
    $$\le 3 \|h\|   \alpha_4 ^{\frac 12}
     \sum_{j=1}^J
     \delta^{\frac
     {q_j -\hat{q}_j}
     {2q_j}
     }
       \E \left (
   \left (
    \int_{\theta_n\wedge \tau_n^k}
   ^{(\theta_n +\delta)\wedge \tau_n^k}
   \|  Z_n (s)) \|_{V_j}^{{q}_j}
  ds \right )^{\frac {\hat{q}_j}
  {2q_j}
  } \right )
   $$
     $$
     + 3 \|h\|   \E \left (
   (1+  \sup_{s\in [0,T]} \|  Z_n (s))   \|^2_H)^{\frac 12}
      \left (
    \int_{\theta_n\wedge \tau_n^k}
   ^{(\theta_n +\delta)\wedge \tau_n^k}
    |g(s)|  
   ds \right )^{\frac 12}
   \right )
   $$
    $$\le 3 \|h\|   \alpha_4 ^{\frac 12}
     \sum_{j=1}^J
     \delta^{\frac
     {q_j -\hat{q}_j}
     {2q_j}
     }
     \left (    \E \left (
   \int_{0}
   ^ T
   \|  Z_n (s)) \|_{V_j}^{{q}_j}
  ds \right )
   \right )
  ^{\frac {\hat{q}_j}
  {2q_j}
  }
   $$
     $$
     + 3 \|h\|  
 \left  (  1+  \E  
 \left ( 
   \sup_{s\in [0,T]} \|  Z_n (s))   \|^2_H  \right )
 \right )  
   ^{\frac 12}
   \left (
       \E \left (    
    \int_{\theta_n\wedge \tau_n^k}
   ^{(\theta_n +\delta)\wedge \tau_n^k}
    |g(s)|  
   ds \right )
   \right ) ^{\frac 12}
   $$
     $$
     \le c_4
     \sum_{j=1}^J
     \delta^{\frac
     {q_j -\hat{q}_j}
     {2q_j}
     }
   +c_4 
   \left (
       \E \left (    
    \int_{\theta_n}
   ^{(\theta_n +\delta)}
    |g(s)|  
   ds \right )
   \right ) ^{\frac 12}.
   $$
  Since 
  $q_j>\hat{q}_j$ and $g\in L^1(0,T)$,
   by the absolute continuity of Lebesgue
   integral, we find that 
  \be\label{tig p9}
  \lim_{\delta \to  0 }   \sup_{n\in \N}  \E \left (
   \left |
    \int_{\theta_n\wedge \tau_n^k}
   ^{(\theta_n +\delta)\wedge \tau_n^k}
   h^* P_nB(s, Z_n (s))Q_n dW
   \right |
   \right ) =0.
\ee
It follows from \eqref{tig p7}-\eqref{tig p9} that
    \be\label{tig p10}
 \lim_{\delta \to 0 }
 \sup_{n\in \N}  \E \left (
  \left |
  \left ( Z_n((\theta_n +\delta)\wedge \tau_n^k )
  -Z_n (\theta_n\wedge \tau_n^k), \  h
  \right )_{(V^*,V)}
  \right | \right ) =0.
\ee
By \eqref{tig p5} and \eqref{tig p10} we get
for every   $\eta>0$,  
$$
 \limsup_{\delta \to 0 } \ 
 \sup_{n\in \N}
  \p
  \left (
  \left |
  \left ( Z_n(\theta_n +\delta)
  -Z_n (\theta_n), \  h
  \right )_{(V^*,V)}
  \right | >\eta
  \right )
$$
$$
\le {\frac 1\eta} \limsup_{\delta \to 0  } \ 
 \sup_{n\in \N}
\E 
  \left (
  \left |
  \left ( Z_n((\theta_n +\delta)\wedge \tau_n^k )
  -Z_n (\theta_n\wedge \tau_n^k), \  h
  \right )_{(V^*,V)}
  \right |  
  \right )
  +  
 \sup_{n\in \N}
    \p
  \left (
   \tau_n^k < T
  \right )
$$
 \be\label{tig p11}
=  
 \sup_{n\in \N}
    \p
  \left (
   \tau_n^k < T
  \right ).
 \ee
 Letting $k\to \infty$, by \eqref{tig p4} we get
 for every   $\eta>0$,  
$$
 \limsup_{\delta \to 0 }\ 
 \sup_{n\in \N}
  \p
  \left (
  \left |
  \left ( Z_n(\theta_n +\delta)
  -Z_n (\theta_n), \  h
  \right )_{(V^*,V)}
  \right | >\eta
  \right )  \le 0,
$$
which implies \eqref{tig p3} and thus completes the proof.
      \end{proof}

    Let $L^\infty_{w^*}
    (0,T; H)$ be the space  
    $L^\infty
    (0,T; H)$ endowed with the  weak-*  topology,
    and
    $L_w^{q_j}(0,T; V_j)$ be the space
    $L^{q_j}(0,T; V_j)$
    endowed with
    the  weak topology.
    The space
     $L_w^{\underline{q}}(0,T; V)$
     is defined similarly.
    Denote by
  \be\label{caly}
   \caly =
    L^{2\underline{q}}(0,T; H)
    \bigcap
     L_w^{\underline{q}}(0,T; V)\bigcap
      L^\infty_{w^*}
    (0,T; H)\bigcap C([0,T], V^*)   
     \bigcap \left (
     \bigcap_{j=1}^J 
     L_w^{q_j }(0,T; V_j)\right )  ,
\ee
which is endowed with the supremum
topology $\calt$ of the corresponding topologies.
In other words,
 $\calt$ is the smallest topology
 on $\caly$  that is
larger than the union of the following topologies
  on $\caly$: 
 $$L^{2\underline{q}}(0,T; H), \ 
     L_w^{\underline{q}}(0,T; V), \ 
       L^\infty_{w^*}
    (0,T; H), 
    \  C([0,T], V^*) ,   \ 
        L_w^{q_j }(0,T; V_j),
        \  j=1,\cdots, J.
        $$
    
    \begin{lem}\label{tig1}
     Suppose   {\bf (H1)}, {\bf (H2)}$^\prime$
and {\bf (H3)}-{\bf (H5)}  hold.
If the embedding $V\subseteq H$ is compact,
then  
    the sequence $\{Z_n \}_{n=1}^\infty$
    is tight in $(\caly, \calt)$.
     \end{lem}
    
    \begin{proof}
    By Lemma \ref{ues1}, we find that there exists
    $c_1=c_1(T)>0$ such that
  $$
    \E \left (
    \sup_{t\in [0,T]} \|Z_n (t)\|_H^2
     +  
    \int_0^T \| Z_n (t) \|_{V}^{\underline{q}} dt
   +\sum_{j=1}^J  
    \int_0^T \| Z_n (t) \|_{V_j}^{q_j} dt
    \right )
    \le c_1.
 $$
    Then for  $\eps>0$, there exists
    $R=R(\eps, T)>0$ such that
    for all $n\in \N$,
    $$
    \p
    \left (
     \sup_{t\in [0,T]} \|Z_n (t)\|_H
     +  \int_0^T \| Z_n (t) \|_{V}^{\underline{q}} dt
     +
     \sum_{j=1}^J  
    \int_0^T \| Z_n (t) \|_{V_j}^{q_j} dt
     >R
    \right ) <{\frac 12} \eps.
    $$
    Let $K_1$ be the   subset  of $\caly$ given by
    $$
    K_1=\left \{
    y\in \caly:  \sup_{t\in [0,T]} \|y (t)\|_H
     +  \int_0^T \| y (t) \|_{V}^{\underline{q}} dt
     +
     \sum_{j=1}^J  
    \int_0^T \| Z_n (t) \|_{V_j}^{q_j} dt
    \le R
    \right \}.
    $$ 
It is clear that
    $\p \left ( Z_n  \in K_1
    \right ) > 1-{\frac 12} \eps$  for all $n\in \N$.
    
    On the other hand, by Lemma \ref{tig} we know
    $\{Z_n\}_{n=1}^\infty $ is tight in 
    in $C([0,T], V^*)$, and thus there exists a compact
    subset $K_2$ of  $C([0,T], V^*)$ such that
     $\p \left ( Z_n  \in K_2
    \right ) > 1-{\frac 12} \eps$  for all $n\in \N$.
    Let $K=K_1\bigcap K_2$. Then
      $\p \left ( Z_n  \in K 
    \right ) > 1 -  \eps$  for all $n\in \N$.
    We will  show  $K$ is a compact
    subset of $(\caly, \calt)$.
    
    Since  $ L^{\underline{q}}(0,T; V)$  is a separable
    reflexive Banach space and $K$ is bounded in
     $ L^{\underline{q}}(0,T; V)$, we see that
     the weak topology of  $ L^{\underline{q}}(0,T; V)$
     on $K$ is metrizable.
     Similarly,
      the weak topology of  $ L^{ {q_j}}(0,T; V_j)$
     on $K$ is also  metrizable
     for every $j=1,\cdots,  J$.
     
     On the other other, since 
     $K$ is bounded in  $L^\infty(0,T; H)
      =( L^1(0,T; H))^*$
     and  $ L^1(0,T; H)$ is separable, we find that
     the weak-*
     topology
     of $ L^\infty(0,T; H)$ is metrizable
     on  $K$.
     Since  $L^{\underline{q}}(0,T; H)
    $ and  $ C([0,T], V^*)$ are also metric spaces,
    we infer that
    $(K, \calt)$ is metrizable, and thus we only need to show
      $(K, \calt)$ is sequentially compact.
     
    Let $\{y_n\}_{n=1}^\infty$ be a sequence in $K$.
    Then $\{y_n\}_{n=1}^\infty$ is bounded in 
     $ L^{\underline{q}}(0,T; V)$,
        $ L^{ {q_j}}(0,T; V_j)$
      and  $ L^\infty(0,T; H)$,
     and hence it has a subsequence
     (not relabeled)
     which is  weakly convergent
     in  $ L^{\underline{q}}(0,T; V)$ 
     and    $ L^{ {q_j}}(0,T; V_j)$,
     and weak-*  convergent
     in  $ L^\infty(0,T; H)$.
     Since $\{y_n\}_{n=1}^\infty$ belongs to $K_2$
      and $K_2$ is compact in $C([0,T], V^*)$, we see
      that $\{y_n\}_{n=1}^\infty$ also has a convergent subsequence
      in $C([0,T], V^*)$.
      Therefore, 
        there exists $y\in K $ such that, up to a subsequence,
      \be\label{tig1 p1}
     \lim_{n\to \infty}
      \| y_n-y\|_{C([0,T], V^*)} =0,
      \ee 
      and $y_n \to y$ in 
      $ L_w^{\underline{q}}(0,T; V)
   $,
           $ L_w^{ {q_j}}(0,T; V_j)$
    and  $ L^\infty_{w^*}
    (0,T; H)$.
    
      It remains to show   $y_n \to y$ in 
      $ L^{2 \underline{q}}(0,T; H)$. 
    Since $y_n,   y  \in K$,  by \eqref{tig1 p1} we have
     $$
     \int_0^T 
      \| y_n (s) -y(s) \|_H^{2\underline{q}}ds
      =
         \int_0^T 
  \left |  ( y_n (s) -y(s), \ y_n (s) -y(s)_{(V^*, V)}
  \right |^{   {\underline{q}} }
    ds
      $$
        $$
        \le  
         \int_0^T  \left (
\|  y_n (s) -y(s)\|_{V^*} \|    y_n (s) -y(s)\|_V
 \right ) ^{   {\underline{q}} }
    ds
    $$
    $$
    \le \|  y_n  -y\|^{   {\underline{q}} }  _{C([0,T], V^*)}
     \int_0^T    \|    y_n (s) -y(s)\|_V
  ^{   {\underline{q}} }
    ds
      $$
       $$
    \le 2^{{\underline{q}} -1 } \|  y_n  -y\|^{   {\underline{q}} }  _{C([0,T], V^*)}
     \int_0^T (   \|    y_n (s) \|_V
  ^{   {\underline{q}} }
  +
   \|    y  (s) \|_V
  ^{   {\underline{q}} })
    ds
      $$
        $$
    \le 2^{{\underline{q}}  }  R  \|  y_n  -y\|^{   {\underline{q}} }  _{C([0,T], V^*)}
    \to 0,
      $$
      as desired.
    \end{proof}

    By Lemma \ref{tig1} and   the Skorokhod-Jakubowski representation theorem on a topological space
    (see Proposition \ref{prop_sj} in Appendix), we have:

\begin{lem} 
\label{tig2}
   Suppose   {\bf (H1)}, {\bf (H2)}$^\prime$
and {\bf (H3)}-{\bf (H5)}  hold.
If the embedding $V\subseteq H$ is compact,
 then
there exist a probability space
 $(\widetilde{\Omega}, \widetilde{\calf},
    \widetilde{\p})$,  
 random variables
 $(\widetilde{{Z}}, \widetilde{W} )$
 and
  $(\widetilde{{Z}}_n, \widetilde{W}_n)$
   defined on $(\widetilde{\Omega}, \widetilde{\calf},
     \widetilde{\p})$ such that 
   in $\caly  
    \times C([0,T], U_0)$, up to a subsequence:

(i) 
For each $n\in \N$, the
 law  of $(\widetilde{Z}_n, \widetilde{W}_n)$
coincides with
    $( {Z}_n, W)$.

(ii)    $(\widetilde{Z}_n, \widetilde{W}_n)\rightarrow (\widetilde{Z}, \widetilde{W} )$,
 $\widetilde{\p}$-almost surely.

(iii)  The sequence $\{\widetilde{Z}_n \}_{n=1}^\infty$
satisfies the uniform estimates: for all $p\ge 1$   and  $n\in \N$,
 $$
  \E \left (
  \sup_{t\in [0,T]}
  \| \widetilde{Z}_n  (t)\|^{2p}_H
  \right  )
  +
   \E \left ( \left (
  \int_0^T \sum_{j=1}^J  \|  \widetilde{Z}_n (t) \|_{V_j}^{q_j} dt
  \right )^{p} \right )
  $$
\be\label{tig2 1}
  + \E \left (  \int_0^{T }
  \| \widetilde{Z}_n  (s)\|_H^{2p-2}  \sum_{j=1}^J
  \| \widetilde{Z}_n  (s)\|_{V_j}^{q_j} ds
  \right ) 
  \le M_3  (1 +  \|x\|^{2p}_H) ,
\ee
  and
$$
  \E \left (
  \sup_{t\in [0,T]}
  \| \widetilde{Z}  (t)\|^{2p}_H
  \right  )
  +
   \E \left ( \left (
  \int_0^T \sum_{j=1}^J  \|  \widetilde{Z} (t) \|_{V_j}^{q_j} dt
  \right )^{p} \right )
  $$
  \be\label{tig2 2}
  + \E \left (  \int_0^{T }
  \| \widetilde{Z}  (s)\|_H^{2p-2}  \sum_{j=1}^J
  \| \widetilde{Z}  (s)\|_{V_j}^{q_j} ds
  \right ) 
  \le M_3  (1 +  \|x\|^{2p}_H) .
\ee
  In addition,
   for all $n\in \N$  and
   $j=1,\cdots, J$,
  \be\label{tig2 3}
     \E \left (
     \int_0^T
   \| A_j (s,     \widetilde{Z} _n(s) ) \|^{\frac {q_j }{q_j  -1}}_{V_j^*}
   ds \right )
   +
    \E \left (
  \int_0^T
  \| P_n B(s,   \widetilde{Z} _n (s)) Q_n \|^{2}_{\call_2(U,H)} ds
  \right )
       \le M_3  (1 +  \|x\|^{\beta+ 2}_H) ,
\ee
  where  $M_3= M_ 3(T, p)>0$  is a constant.
  \end{lem}

 \begin{proof}
 To apply Proposition  \ref{prop_sj}
 on the topological space
 $\caly  
    \times C([0,T], U_0)$, we need to find 
   a sequence of   continuous real-valued
functions
defined
on $\caly  
    \times C([0,T], U_0)$
 that  separate  points. It suffices to find such a sequence
in  $\caly$  and $ C([0,T], U_0)$ separately.
 Since $ C([0,T], U_0)$ is a
  separable Banach space,
 there exists  a sequence of   continuous real-valued
functions
defined
on $    C([0,T], U_0)$   that  separate  points.
  
  For the space $(\caly, \calt)$,  note
  that  $\calt$ is 
  the supremum
topology   of the corresponding topologies
in $\caly$ and hence it is stronger than any  topology of 
 $$ 
    L^{2\underline{q}}(0,T; H), \ 
    L_w^{\underline{q}}(0,T; V),\ 
            L_w^{ {q_j}}(0,T; V_j) , \ 
  L^\infty_{w^*}
    (0,T; H)   \quad \text{and }\  C([0,T], V^*).
    $$
 Let 
  $\{f_j^*\}_{j=1}^\infty$
  be a 
  dense sequence  
  in  $(  L^{\underline{q}}(0,T; V))^*$.
  Then $\{f_j^*\}_{j=1}^\infty$
  is  a sequence
of continuous
functions
on  $ L_w^{\underline{q}}(0,T; V)$
that separate points.
Since $\calt$ is     stronger than
the weak topology of $   L^{\underline{q}}(0,T; V) $,
we infer that
$f_j^*: 
 (\caly, \calt) \to \R$ is also continuous and thus
  $\{f_j^*\}_{j=1}^\infty$
forms   a sequence
of continuous
functions on $ (\caly, \calt) $
that separate points.

 Then by Lemma  \ref{tig1} and Proposition
 \ref{prop_sj} we obtain (i) and (ii) immediately.
 The uniform estimate \eqref {tig2 1}
 follows from  (i) and Lemma \ref{ues1},
 while
  \eqref {tig2 2} follows from (i),  \eqref {tig2 1}
  and the lower semicontinuity 
  of $\| \cdot\|_H$ and $\|\cdot \|_V$ on $V^*$.
  Finally, the   uniform estimate \eqref {tig2 3}
 follows from  (i) and  the argument of Lemma \ref{ues2}.
   \end{proof}

    Given $n\in \N$, let
    $(\widetilde{\calf}^n_t)_{t\in [0,T]}$ be the filtration
    which satisfies   the usual condition   generated by
    $$
    \{  \widetilde{Z}_n (s), \ 
    \widetilde{W}_n (s): \ s\in [0, t]\}.
    $$
    Similarly,
     let
    $(\widetilde{\calf}_t)_{t\in [0,T]}$ be the filtration
    which satisfies   the usual condition   generated by
    $$
    \{  \widetilde{Z} (s), \ 
    \widetilde{W} (s): \ s\in [0, t]\}.
    $$
    Then $   \widetilde{W}^n $
    is an  $(\widetilde{\calf}^n_t)$-cylindrical Wiener process, 
    and
     $   \widetilde{W} $
    is an  $(\widetilde{\calf}_t)$-cylindrical Wiener process 
    on  $U$, respectively.
    
    Since $Z_n$ satisfies \eqref{ode1}, by Lemma
    \ref{tig2}  (i) we find  that
    $ \widetilde{Z}_n$ satisfies
   \be\label{ode1n}
     \widetilde{Z}_n (t)
     = P_n x
     +\int_0^t P_n A(s,
      \widetilde{Z}_n (s)) ds
      +
      \int_0^t P_n 
      B(s,  \widetilde{Z}_n (s)) Q_n d
       \widetilde{W}_n (s),
       \ee
    which can be proved by the argument of    \cite{ben1}.
    Next, we prove $
    (\widetilde{\Omega},
    \widetilde{\calf}, ( \widetilde{\calf}_t)_{t\in [0,T]},
    \widetilde{\p}, \widetilde{Z} , \widetilde{W})
    $ is a martingale solution of \eqref{sde1}-\eqref{sde2}.

    \subsection {Proof of Theorem \ref{main}}
    We are now in a position to 
     complete the proof of the main  result
    of this section.

    {\bf Step (i)}: Weak convergence.
     By Lemma \ref{tig2}, we infer that
    there exist 
    $\widetilde{B} \in 
      L^{2}
    ([0,T] \times \widetilde{\Omega}, 
    \call_2(U,H))$ and 
    $ \widetilde{A}_j \in
     \ L^{{\frac {q_j}{q_j -1}}}
    ([0,T] \times \widetilde{\Omega}, V^*_j)$
    for every $ j=1,\cdots, J$,   such that,
    up to a subsequence,
    \be\label{ms1 p1}
    \widetilde{Z}_n
    \to \widetilde{Z}
    \ \text{ weakly   in   } \ L^{q_j}
    ([0,T] \times \widetilde{\Omega}, V_j), \quad \forall \  j=1,\cdots, J,
      \ee
       \be\label{ms1 p2}
   A_j (\cdot,  \widetilde{Z}_n)
    \to \widetilde{A}_j
    \ \text{ weakly   in   } \ L^{{\frac {q_j}{q_j -1}}}
    ([0,T] \times \widetilde{\Omega}, V^*_j), \quad \forall \  j=1,\cdots, J.
      \ee
        \be\label{ms1 p3}
        P_n B(
   \cdot,  \widetilde{Z}_n) Q_n 
    \to \widetilde{B}
    \ \text{ weakly   in   } \ 
      L^{2}
    ([0,T] \times \widetilde{\Omega}, 
    \call_2(U,H)).
      \ee
     On the other hand, by Lemma \ref{tig2} (ii)
     and \eqref{caly} we have,
     $ \widetilde{\p}$-almost surely,
       \be\label{ms1 p4}
  \lim_{n\to \infty}
  \left (   \| \widetilde{Z}_n
 - \widetilde{Z}\|_{ L^{2\underline{q}}(0,T; H)}
+  \| \widetilde{Z}_n
 - \widetilde{Z}\|_{C([0,T], V^*)} \right )
= 0,
      \ee
            \be\label{ms1 p4a}
    \widetilde{Z}_n
    \to \widetilde{Z}
    \ \text{ weakly  \   in   } \ 
         L^{ {q_j}}(0,T; V_j),
         \quad \forall \  j=1, \cdots, J,
         \ee 
        and
                \be\label{ms1 p5}
    \widetilde{Z}_n
    \to \widetilde{Z}
    \ \text{ weak-*  \   in   } \ 
          L^\infty
    (0,T; H).
         \ee
By
\eqref{ms1 p4a} and
 \eqref{ms1 p5} we infer that  $ \widetilde{\p}$-almost surely,
   \be\label{ms1 p6}
  \sup_{1\le j\le J } 
  \   \sup_{n\in \N}
     \|   \widetilde{Z}_n
      \|_{  L^{ {q_j}}(0,T; V_j)}
      <\infty
      \quad \text{and}
      \quad
  \sup_{n\in \N}
   \|  \widetilde{Z}_n\|_{  L^\infty
    (0,T; H)} <\infty .
     \ee

       {\bf Step (ii)}:  Prove 
      $  \widetilde{B} =   B(
   \cdot,  \widetilde{Z})$ a.e. on  
   $ [0,T] \times  \widetilde{\p} $, and
    \be\label{ms1 p6a}
        P_n B(
   \cdot,  \widetilde{Z}_n) Q_n 
    \to \widetilde{B}
    \ \text{ strongly   in   } \ 
      L^{2}
    ([0,T] \times \widetilde{\Omega}, 
    \call_2(U,H) ).
      \ee

   Note that 
   $$ \| P_n B(
   \cdot,  \widetilde{Z}_n)Q_n
   - 
    B(
   \cdot,  \widetilde{Z})
   \|^2_{
   L^2([0,T]\times  \widetilde{\Omega},
   \call_2(U,H)  )
   }
   $$
    $$ 
    \le 3
    \| P_n  \left ( B(
   \cdot,  \widetilde{Z}_n) 
   -   B(
   \cdot,  \widetilde{Z})
   \right ) Q_n
   \|^2_{
   L^2([0,T]\times  \widetilde{\Omega},
   \call_2(U,H)  )
   }
   $$
    $$ 
   + 3
    \| P_n B(
   \cdot,  \widetilde{Z})
   \left ( Q_n
   -  I \right ) 
   \|^2_{
   L^2([0,T]\times  \widetilde{\Omega},
   \call_2(U,H)  )
   }
   +   3
    \| \left ( P_n -I \right )
     B(
   \cdot,  \widetilde{Z})
   \|^2_{
   L^2([0,T]\times  \widetilde{\Omega},
   \call_2(U,H)  )
   }
   $$
    $$ 
    \le 3
    \|    B(
   \cdot,  \widetilde{Z}_n) 
   -   B(
   \cdot,  \widetilde{Z})
   \|^2_{
   L^2([0,T]\times  \widetilde{\Omega},
   \call_2(U,H)  )
   }
   $$
\be\label{ms1 p8}
   + 3
    \| P_n B(
   \cdot,  \widetilde{Z})
   \left ( Q_n
   -  I \right ) 
   \|^2_{
   L^2([0,T]\times  \widetilde{\Omega},
   \call_2(U,H)  )
   }
   +   3
    \| \left ( P_n -I \right )
     B(
   \cdot,  \widetilde{Z})
   \|^2_{
   L^2([0,T]\times  \widetilde{\Omega},
   \call_2(U,H)  )
   }.
   \ee
   For the first term on the right-hand side of \eqref{ms1 p8},
   by \eqref{ms1 p4},  \eqref{ms1 p6} and
 \eqref{h5c}  we get,
    $\widetilde{\p}$-almost surely,
    $$ \int_0^T \|    B(t
   ,  \widetilde{Z}_n (t) ) 
   -   B(
   t,  \widetilde{Z} (t) )
   \|^2_{ 
   \call_2(U,H) 
   } dt 
   \to 0,
   $$
   which along with
   \eqref{tig2 1}-\eqref{tig2 2}
   and \eqref{h5a}   implies that
   \be\label{ms1 p9}
  \lim_{n\to \infty}   \|    B(
   \cdot,  \widetilde{Z}_n) 
   -   B(
   \cdot,  \widetilde{Z})
   \| _{
   L^2([0,T]\times  \widetilde{\Omega},
   \call_2(U,H)  )
   } =0.
\ee
On the other hand, since 
 $B(
   \cdot,  \widetilde{Z})
 \in 
   L^2([0,T]\times  \widetilde{\Omega},
   \call_2(U,H)  )$, we find that
   the last  two terms on the
   right-hand side of \eqref{ms1 p8}
   converge to zero, and thus
   by \eqref{ms1 p8}-\eqref{ms1 p9} we obtain
\be\label{ms1 p10}
\lim_{n\to \infty}
 \| P_n B(
   \cdot,  \widetilde{Z}_n)Q_n
   - 
    B(
   \cdot,  \widetilde{Z})
   \|^2_{
   L^2([0,T]\times  \widetilde{\Omega},
   \call_2(U,H)  )
   }
 =0.
 \ee
 Then \eqref{ms1 p6a} follows from 
    \eqref{ms1 p3} and \eqref{ms1 p10}
immediately.
   
  By  Lemma \ref{tig2} (ii) and 
     Step (ii),    it follows from
     Lemma 2.1 in \cite{deb1} that
   \be\label{ms1 p11}
     \lim_{n\to \infty}
       \int_0^t P_n 
      B(s,  \widetilde{Z}_n (s)) Q_n d
       \widetilde{W}_n (s)
       =  \int_0^t  
      B(s,  \widetilde{Z}  (s))  d
       \widetilde{W}  (s),
    \ee
       in probability in $L^2(0,T; H)$.
       Then letting $n\to \infty$ in \eqref{ode1n},
       by \eqref{ms1 p1}-\eqref{ms1 p2}
       and \eqref{ms1 p11}
       we find  that
       for almost all $(t,\omega) \in [0,T] \times   \widetilde{\Omega} $,
        \be\label{ms1 p12}
     \widetilde{Z}  (t)
     =   x
     +\int_0^t \sum_{j=1}^J  A_j (s) ds
      +
      \int_0^t 
      B(s,  \widetilde{Z}  (s))   d
       \widetilde{W}  (s) \quad \text{in } \ V^*.
       \ee
       By \eqref{tig2 2} and the argument of \cite{kry1}
       we infer that
       there exists an $H$-valued continuous
       $(\widetilde{\calf}_t)$-adapted process
       $\hat{Z}$ such that
       $\hat{Z} =   \widetilde{Z}$ almost everywhere
       on $[0,T] \times   \widetilde{\Omega}$ and
       $\hat{Z} (t)$ is equal to the right-hand side of
       \eqref{ms1 p12} for all $t\in [0,T]$,
    $ \widetilde{\p}$-almost surely.
   From now on, we identify $\hat{Z}$ with
  $ \widetilde{Z}$, and thus \eqref{ms1 p12}
  holds   for all $t\in [0,T]$,
    $ \widetilde{\p}$-almost surely.

      {\bf Step (iii)}: Prove:
     \be\label{ms1 p12a}
      \liminf_{n\to \infty}
       \widetilde{\E} \left (
       \int_0^T \sum_{j=1}^J
       (A_j (t,  \widetilde{Z}_n (t)),
       \  \widetilde{Z}_n (t) )_{(V_j^*, V_j)} dt
       \right )
   \ge
    \widetilde{\E} \left (
       \int_0^T \sum_{j=1}^J
       (  \widetilde{A}_j  (t) ,
       \  \widetilde{Z} (t) )_{(V_j^*, V_j)} dt
       \right ).
    \ee
       
       By \eqref{ode1n},  \eqref{ms1 p12}
       and It\^{o}'s formula, we have
    $$
         \widetilde{\E} \left (
         \| \widetilde{Z}_n (t)\|_H^2
         \right )
         = \|P_n x\|^2
         +2 
        \widetilde{\E} \left (
       \int_0^T \sum_{j=1}^J
       (A_j (t,  \widetilde{Z}_n (t)),
       \  \widetilde{Z}_n (t) )_{(V_j^*, V_j)} dt
       \right )
       $$
         \be\label{ms1 p13}
       +  \widetilde{\E} \left (
       \int_0^T \|P_n B(t, \widetilde{Z}_n (t)  )Q_n
       \|^2_{\call_2(U,H)} dt
       \right ),
   \ee
   and
    $$
         \widetilde{\E} \left (
         \| \widetilde{Z} (t)\|_H^2
         \right )
         = \| x\|^2
         +2 
        \widetilde{\E} \left (
       \int_0^T \sum_{j=1}^J
       ( \widetilde{A}_j  (t) ,
       \  \widetilde{Z} (t) )_{(V_j^*, V_j)} dt
       \right )
       $$
         \be\label{ms1 p14}
       +  \widetilde{\E} \left (
       \int_0^T \|B(t, \widetilde{Z} (t)  )
       \|^2_{\call_2(U,H)} dt
       \right ).
   \ee
      Note that  $ \widetilde{Z}_n
       \to  \widetilde{Z}$ in $C([0,T], V^*)$,
       which along with the lower semicontinuity
       of $\| \cdot \|_H$ on $V^*$ and the Fatou lemma
       implies that
       \be\label{ms1 p15}
   \liminf_{n\to \infty}
     \widetilde{\E} \left (
         \| \widetilde{Z}_n  (t)\|_H^2
         \right )
         \ge
          \widetilde{\E} \left ( \liminf_{n\to \infty}
         \| \widetilde{Z}_n  (t)\|_H^2
         \right )
         \ge    \widetilde{\E} \left ( 
         \| \widetilde{Z}   (t)\|_H^2
         \right ).
         \ee
       Then \eqref{ms1 p12a}
          follows from \eqref{ms1 p6a}
         and \eqref{ms1 p13}-\eqref{ms1 p15}.

       {\bf Step (iv)}:  Prove 
      $ \sum_{j=1}^J
       \widetilde{A}_j  =  \sum_{j=1}^J  A_j (
   \cdot,  \widetilde{Z})$  a.e. on  
   $ [0,T] \times  \widetilde{\p} $.
   
   We follow the argument as in \cite{roc1}
   with  modifications.
   By \eqref{h3a} and \eqref{h4a} we have
   $$
   \sum_{j=1}^J
   (A_j (t, \widetilde{Z}_n (t)),
   \widetilde{Z}_n (t) - \widetilde{Z}(t))_{(V_j^*,
   V_j)}
   $$
   $$
   \le
  -  {\frac 12} \alpha_2
  \sum_{j=1}^J
  \|  \widetilde{Z}_n (t)\|_{V_j}^{q_j}
  +{\frac 12} g(t) (1+ \|  \widetilde{Z}_n (t)\|^2_H)
  +
   \sum_{j=1}^J
  \| A_j (t, \widetilde{Z}_n (t) )\|_{V_j^*}
  \|  \widetilde{Z}(t) \|_{V_j}
   $$
    $$
   \le
  -  {\frac 12} \alpha_2
  \sum_{j=1}^J
  \|  \widetilde{Z}_n (t)\|_{V_j}^{q_j}
  +{\frac 12} g(t) (1+ \|  \widetilde{Z}_n (t)\|^2_H)
  $$
  $$+
   \sum_{j=1}^J
\left (
\left  (\alpha_3 \|  \widetilde{Z}_n (t)\|_{V_j}^{q_j}
+ g(t) \right  )
\left  (1+ \|  \widetilde{Z}_n (t)\|_H^\beta \right )
\right )^{\frac {q_j -1}{q_j}}
  \|  \widetilde{Z}(t) \|_{V_j}
   $$
    $$
   \le
  -  {\frac 14} \alpha_2
  \sum_{j=1}^J
  \|  \widetilde{Z}_n (t)\|_{V_j}^{q_j}
  +  g(t) (1+ \|  \widetilde{Z}_n (t)\|^2_H)
  $$
  \be\label{ms1 p20}
  + c_1
   \sum_{j=1}^J
   (1+   
  \|  \widetilde{Z}_n (t)\|_H^{\beta (q_j -1)} )
   \|  \widetilde{Z}(t) \|_{V_j}^{q_j},
\ee
where $c_1>0$ is a constant depending only on
$q_j$ for $j=1,\cdots, J$.
Denote by
   \be\label{ms1 p21}
  \varphi_n (t, \omega)  = \sum_{j=1}^J
   (A_j (t, \widetilde{Z}_n (t, \omega)),
   \widetilde{Z}_n (t, \omega) - \widetilde{Z}(t, \omega ))_{(V_j^*,
   V_j)}
\ee
   and 
   \be\label{ms1 p22}
   \psi_n (t, \omega) =
      g(t) (1+ \|  \widetilde{Z}_n (t, \omega )\|^2_H)
  + c_1
   \sum_{j=1}^J
   (1+   
  \|  \widetilde{Z}_n (t, \omega)\|_H^{\beta (q_j -1)} )
   \|  \widetilde{Z}(t, \omega) \|_{V_j}^{q_j}.
\ee
By \eqref{ms1 p20}-\eqref{ms1 p22} we get
   \be\label{ms1 p23}
\varphi_n (t, \omega)
\le -  {\frac 14} \alpha_2
  \sum_{j=1}^J
  \|  \widetilde{Z}_n (t)\|_{V_j}^{q_j}
  +\psi_n (t,\omega).
  \ee

 We now prove  the sequence 
  $\{\psi_n\}_{n=1}^\infty$ is uniformly
  integrable on $[0,T] \times \widetilde{\Omega}$.
   By  \eqref{tig2 1}, \eqref{tig2 2}
   and \eqref{ms1 p4} we get 
   $$
  \lim_{n\to \infty}
    \|  \widetilde{Z}_n 
   -\widetilde{Z}   \|_{L^1([0,T] \times \widetilde{\Omega}, H) } 
   =0, 
   $$
   and hence there exists a subsequence
   (not relabeled) such that
   for almost all $(t, \omega)
   \in [0,T] \times \widetilde{\Omega}$,
 \be\label{ms1 p7}
    \lim_{n\to \infty}  \|  \widetilde{Z}_n (t, \omega)
   -\widetilde{Z} (t, \omega)
   \|_H  =0.
   \ee
   By
   \eqref{ms1 p22}  and 
   \eqref{ms1 p7}    we have,
  for almost all $(t,\omega) \in [0,T] \times \widetilde{\Omega}$,
   \be\label{ms1 p23}
   \psi_n (t, \omega) \to \psi (t,\omega),
   \ee
   where 
 \be\label{ms1 p23a}
   \psi(t,\omega) =
      g(t) (1+ \|  \widetilde{Z} (t, \omega )\|^2_H)
  + c_1
   \sum_{j=1}^J
   (1+   
  \|  \widetilde{Z} (t, \omega)\|_H^{\beta (q_j -1)} )
   \|  \widetilde{Z}(t, \omega) \|_{V_j}^{q_j}.
\ee
By \eqref{ms1 p6}  and \eqref{ms1 p22} we get
  \be\label{ms1 p24a}
   \psi_n (t, \omega)
    \le
      g(t) (1+  \|  \widetilde{Z}_n (\cdot, \omega )\|^2_
      {L^\infty(0,T; H)} )
  + c_1
   \sum_{j=1}^J
   (1+   
  \|  \widetilde{Z}_n (\cdot, \omega)\|_{
  L^\infty (0,T; H)} ^{\beta (q_j -1)} )
   \|  \widetilde{Z}(t, \omega) \|_{V_j}^{q_j},
\ee
\be\label{ms1 p24}
 \le
      g(t) (1+  c_2(\omega)  )
  + c_1
   \sum_{j=1}^J
   (1+   c_2  ^{\beta (q_j -1)}(\omega)   )
   \|  \widetilde{Z}(t, \omega) \|_{V_j}^{q_j},
  \ee
   where $c_2(\omega) = 
   \sup_{n\in \N}  \|  \widetilde{Z}_n (\cdot, \omega )\|_
      {L^\infty(0,T; H)} <\infty$.
      By \eqref{tig2 2} we see that
      the right-hand side of \eqref{ms1 p24}
      in integrable with respect to $t$
       on $[0,T]$,  $\widetilde{\p}$-almost surely,
      which along with \eqref{ms1 p23}
      and the Lebesgue dominated convergence
      theorem  implies that,
        $\widetilde{\p}$-almost surely,
      \be\label{ms1 p25}
      \int_0^T |\psi_n (t, \omega) -\psi(t,\omega)| dt
      \to 0.
      \ee
 
  Next, we show  the left-hand side of
  \eqref{ms1 p25}  
   is uniformly integrable
   with respect to $\omega$
    on  $\widetilde{\Omega}$.
  Indeed, for every $p>1$,  
  we have
  $$
   \widetilde{\E}
   \left (
   \left |
   \int_0^T |\psi_n (t, \omega) -\psi(t,\omega)| dt
   \right |^p
   \right )
   $$
\be\label{ms1 p26}
   \le 2^{p-1}
   \widetilde{\E}
   \left (
   \left |
   \int_0^T
    |\psi_n (t, \omega)|   dt
   \right |^p
   + \left | \int_0^T
    |\psi (t, \omega)|   dt
   \right |^p
     \right ).
     \ee
     For the first term on the right-hand side
     of \eqref{ms1 p26},
  by \eqref{tig2 1}-\eqref{tig2 2}
  and \eqref{ms1 p24a} we get
  for all $n\in \N$,
$$
   \widetilde{\E}
   \left (
   \left |
   \int_0^T
    |\psi_n (t, \omega)|   ds
   \right |^p \right )
    \le
c_3
   \left (\int_0^T
    |g(t)| dt
    \right )^p   \widetilde{\E}
   \left (
   1+  \sup_{t\in [0,T]} \|  \widetilde{Z}_n (t )\|^{2p}_H
     \right  )  
     $$
     $$
  + c_3 \widetilde{\E}
   \left (
   \sum_{j=1}^J
  \left  (1+   \sup_{t\in [0,T]}
  \|  \widetilde{Z}_n (t)\| _H
  ^{\beta p  (q_j -1)} \right )
  \left (
  \int_0^T  \|  \widetilde{Z}(t) \|_{V_j}^{q_j}
  \right )^p
  \right )
   $$
   $$
     \le
c_3
   \left (\int_0^T
    |g(t)| dt
    \right )^p   \widetilde{\E}
   \left (
   1+  \sup_{t\in [0,T]} \|  \widetilde{Z}_n (t )\|^{2p}_H
     \right  )  
     $$
  \be\label{ms1 p27}
  + 2 c_3 
   \sum_{j=1}^J  \widetilde{\E} 
   \left (
 1+   \sup_{t\in [0,T]}
  \|  \widetilde{Z}_n (t)\| _H
  ^{2 \beta p  (q_j -1)}   \right )
  \widetilde{\E} \left (
  \left (
  \int_0^T  \|  \widetilde{Z}(t) \|_{V_j}^{q_j}
  \right )^{2p} 
  \right ) 
  \le c_4,
\ee
  where $c_4=c_4(p, x)>0$ is a constant.
  Similarly,
   for  the second  term on the right-hand side
     of \eqref{ms1 p26},
  by \eqref{tig2 2}
  and \eqref{ms1 p23a} we get
$$
   \widetilde{\E}
   \left (
   \left |
   \int_0^T
    |\psi  (t, \omega)|   ds
   \right |^p \right )
    \le 
c_5
   \left (\int_0^T
    |g(t)| dt
    \right )^p   \widetilde{\E}
   \left (
   1+  \sup_{t\in [0,T]} \|  \widetilde{Z} (t )\|^{2p}_H
     \right  )  
     $$
 \be\label{ms1 p28}
  + 2 c_3 
   \sum_{j=1}^J  \widetilde{\E} 
   \left (
 1+   \sup_{t\in [0,T]}
  \|  \widetilde{Z} (t)\| _H
  ^{2 \beta p  (q_j -1)}   \right )
  \widetilde{\E} \left (
  \left (
  \int_0^T  \|  \widetilde{Z}(t) \|_{V_j}^{q_j}
  \right )^{2p} 
  \right )<\infty.
\ee
By \eqref{ms1 p26}-\eqref{ms1 p28} we obtain,
  for every $p>1$,  
  we have
  $$
   \widetilde{E}
   \left (
   \left |
   \int_0^T |\psi_n (t, \omega) -\psi(t,\omega)| dt
   \right |^p
   \right )<\infty,
   $$
 and hence    
  $  \int_0^T |\psi_n (t, \omega) -\psi(t,\omega)| dt
$ is uniformly integrable
with respect to $\omega$ on $\widetilde{\Omega}$,
which along with \eqref{ms1 p25}
shows that
   \be\label{ms1 p29}
      \widetilde{E}
   \left (  \int_0^T |\psi_n (t, \omega) -\psi(t,\omega)| dt
   \right )
      \to 0.
      \ee
      By \eqref{ms1 p29} we see that
      $\psi_n \to \psi$ in $L^1([0,T]\times
      \widetilde{\Omega})$ and thus
      $\{\psi_n\}_{n=1}^\infty$
      is uniformly integrable on
      $ [0,T]\times
      \widetilde{\Omega}$.
      
      Based on the   uniform integrability
       of $\{\psi_n\}_{n=1}^\infty$,
       by the 
       pseudo-monotone 
       argument of Lemma 2.16 in \cite{roc1},
       one  can prove   $ \sum_{j=1}^J
       \widetilde{A}_j  =  \sum_{j=1}^J  A_j (
   \cdot,  \widetilde{Z})$  a.e. on  
   $ [0,T] \times  \widetilde{\p} $.
   The details are omitted here.
   Then by \eqref{ms1 p12} we find that
  $ \widetilde{Z}$ is a martingale solution
  of  \eqref{sde1}-\eqref{sde2}.
  
  If, in addition, {\bf (H2)} is fulfilled, then 
  the solutions of \eqref{sde1}-\eqref{sde2}
  are pathwise unique, which follows from the
  standard argument (see,  e.g.,  \cite{liu1, roc1}).
This   completes   the proof of Theorem \ref{main}.

      \begin{rem}
      The uniform estimate
      \eqref{ms1 p6} is crucial
      for proving the convergence
      \eqref{ms1 p25}
      and hence the uniform integrability
      of $\{\psi_n\}_{n=1}^\infty$
      on  $ [0,T]\times
      \widetilde{\Omega}$, which is 
      a consequence of \eqref{ms1 p29}.
      Without \eqref{ms1 p6},    it is
      difficult (if it is not  impossible)  to obtain the
      uniform  integrability
      of $\{\psi_n\}_{n=1}^\infty$.
      To establish the uniform estimate
      \eqref{ms1 p6}, we first show the
      tightness of the  sequence
      $\{{Z}_n\}_{n=1}^\infty$
      in $L^\infty_{w^*} (0,T)$
      and $L^{q_j}_w (0,T; V_j)$, 
       and then apply the 
       Skorokhod-Jakubowski
       representation theorem
       in a topological space  
       instead of  a metric space. 
       The classical Skorokhod
       representation theorem
       in a metric space  does not apply
       to   $L^\infty_{w^*} (0,T)$
       or $L^{q_j}_w (0,T; V_j)$.
      \end{rem}

   \newpage
   Note that 
   if $A(t,\cdot): V\to V^*$ is pseudo-monotone for 
each $t\in [0,T]$,
then it follows from  {\bf (H4)}
that  $A(t,\cdot): V\to V^*$ is 
demicontinuous  for 
each $t\in [0,T]$
(see, e.g., Remark 5.2.12 in \cite{liu1}).
In this case,
from  the proof of Theorem \ref{main},
we can also obtain  the existence of martingale
solutions  to  \eqref{sde1}-\eqref{sde2}
when 
condition {\bf (H2)}$^\prime$  is 
  further  weakened by the
pseudo-monotonicity of the operator $A$.
More precisely, we have:

 \begin{thm}\label{main1}
 Suppose {\bf (H1)}
and {\bf (H3)}-{\bf (H5)} are fulfilled.
If  the embedding $V \subseteq H$ is compact
and $A(t,\cdot): V\to V^*$ is pseudo-monotone for 
each $t\in [0,T]$,
then
for every $x\in H$,  \eqref{sde1}-\eqref{sde2}
has at least one martingale solution 
which  satisfies the uniform estimates given by
 \eqref{main 1}. 
  \end{thm}

\section{Martingale solutions of fractional 
reaction-diffusion  equations}
 \setcounter{equation}{0}

In this section, we apply Theorem \ref{main}
to investigate the existence of martingale
solutions of the fractional stochastic
reaction-diffusion equation
\eqref{rde1}-\eqref{rde3}
defined in a bounded domain
$\o$ in $\R^n$
driven by continuous superlinear noise.

 We start with the definition
 of the fractional Laplace operator.
  Given $s\in (0,1)$, 
  the operator
 $ (-\Delta )^s$ is defined by,
for every  $u\in L^2(\R^n)$,
$$
(-\Delta )^s u
= 
  {\mathcal{F}^{-1}}
(|\xi|^{2 s} (\mathcal{F} u)), \quad \xi \in \R^n,
$$
where
${\mathcal{F}}$    and 
${\mathcal{F}}^{-1}$ are 
the Fourier transform   and the 
inverse 
 Fourier transform, respectively.
 The   fractional Sobolev space
  $H^ s  (\R^n)$ 
is defined  by
  $$
 H^s(\R^n)
 =\left \{
 u\in L^2(\R^n): (-\Delta)^{\frac {s}2}
 u \in L^2(\R^n) \right \},
  $$
  with norm
  $$
 \| u\|_{H^ s (\R^n)}^2
 = \| u\|^2_{L^2(\R^n)}
 +   \| (-\Delta)^{{\frac  s2}} u \|^2_{L^2(\R^n)}, 
$$
and  inner product
 $$
 (u, v)_{H^s (\R^n)}
 =
  \int_{\R^n} u(x)v(x)  dx 
  +
  \int_{\R^n}
  \left ( (-\Delta)^{{\frac  s2}} u (x)
  \right )
  \left (
    (-\Delta)^{{\frac  s2}} v (x)
    \right ) dx,
    $$ 
  for all
  $u, v\in H^s
  (\R^n)$.
  As demonstrated in 
    \cite{dine1}   
 the norm   of
$H^s ({\R^n})$ can be
calculated as follows:
   $$
 \| u\|^2_{H^ s ({\R^n}) }
 = 
  \| u\|^2_{L^2(\R^n)}
  +
  {\frac {1}2}C(n,s)
  \int_{\R^n} \int_{\R^n}
 {\frac {|u(x)- u(y)|^2} {|x-y|^{n+2 s} }} dxdy , 
$$
 for
    $u, v \in H^ s  ({\R^n})$, 
  where  
    $C(n, s) 
 =\frac{s 4^s\Gamma(\frac{n+2s}{2})}
{\pi^\frac{n}{2}\Gamma(1-s)}
$.

In the sequel,  we write  
$H= \{ u\in 
  L^2(\R^n):  u=0 \ \text{a.e.  on }  \R^n \setminus \o
  \}$ with norm
  $\| \cdot \|_H$ and inner product
  $(\cdot, \cdot )_H$,
 $V_1=   \{ u\in 
  H^s(\R^n):  u=0 \ \text{a.e.  on }  \R^n \setminus \o
  \}$, 
  $V_2 =  \{ u\in 
  L^p(\R^n):  u=0 \ \text{a.e.  on }  \R^n \setminus \o
  \}$ and   $V_3=H$.
  Denote by 
  $V= V_1 \bigcap V_2 \bigcap V_3=V_1 \bigcap V_2 $.
Then we have
$V \subseteq  H  \equiv  H^*  \subseteq  V^*$.
Since $\o$ is  bounded, we find that
the embedding $V \subseteq  H$  is compact.
Moreover, the norm $\| (-\Delta)^{\frac s2} \cdot \|_{L^2
(\R^n)}$ is equivalent to the norm
$\|\cdot \|_{H^s(\R^n)}$ on $V_1$.

     Throughout this section,
     we assume 
 $f: \R \times \R^n  \times \R
 \to \R$ is a continuous function such that
 for all $t, u, u_1, u_2 \in \R$ and $x\in \R^n$,
 $f(t,x, 0) =0$ and 
 \be\label{f1}
 f(t,x, u_1) \le f(t,x, u_2) \quad  \text{ if } \ u_2\le u_1,
 \ee
 \be\label{f2}
 f(t,x, u) u \ \le -  \delta_1 |u|^p +  \varphi_1 (t,x),
 \ee
  \be\label{f3}
 | f(t,x, u) |  \ \le \delta_2  |u|^{p-1} +  \varphi_2 (t,x),
 \ee
 where $\delta_1>0$,
 $\delta_2>0$ and $p> 2$ are constants,
 $\varphi_1\in L^1([0,T] \times \R^n) $
 and $\varphi_2 \in L^{\frac p{p-1}}
 ([0,T] \times \R^n)$ for every $T>0$.

 For the nonlinear term 
 $h$, we assume  that
 $h: \R \times \R^n  \times \R
 \to \R$ is   continuous   such that
 for all $t, u, u_1, u_2 \in \R$ and $x\in \R^n$,
 $h(t,x, 0) =0$ and 
 \be\label{h1}
 |h(t,x, u_1) -  h(t,x, u_2)|
 \le \varphi_3 (t,x)
 |u_1-u_2|  ,
 \ee
where $\varphi_3\in L^1(0,T; L^\infty(\R^n))$
 for every $T>0$.

For every $i\in \N$, 
let
        $\sigma_i  :
      \R \times \R^n \times \R
      \to \R$  be a mapping given  by   
  \be\label{sig1}
     \sigma_i  (t,x, u)
     = \sigma_{1,i}   (t,x)
     +   
         \sigma_{2, i } ( u)  
        , 
       \quad \forall \ t\in \R, \ x\in \R^n, \ u\in \R,
\ee
     where      
   $\sigma_{1,i}: [0,T]\to L^2(\R^n)$ for every $T>0$
    such that
   \be\label{sig2}
   \sum_{i=1}^\infty  \| \sigma_{1,i}\|^2
   _ {L^2(0,T; L^2(  \R^n) )}
   <\infty.
   \ee 
    Assume that  
    for each $i\in \N$,
    $\sigma_{2,i}: \R
    \to \R$ is continuous such that
  there exist  positive numbers $\beta_i$ and $\gamma_i$
   such that   for all 
    $u  \in \R$,
    \be\label{sig3}
   |\sigma_{2,i}  (u)  |^2
   \le \beta_i    +\gamma _i  |u|^q,
   \ee
   where $q\in [2, p)$ and
  \be\label{sig4}
  \sum_{i=1}^\infty  ( 
    \beta_i  + \gamma_i  ) <\infty .
     \ee
     
     By \eqref{sig3} we see  that
     $   \sigma_{2,i}$
     is allowed   to have a   superlinear growth
     $q$ as long as $q<p$. 
  Given $v\in V_2$  and $t\in \R$, define an operator
  $B (t, v): l^2\to H$ by
\be\label{sig4a}
  B (t,v) (u) (x)
  =\sum_{i=1}^\infty
  \sigma_i (t,x, v(x)) u_i
   \quad \forall \ 
   u=\{u_i\}_{i=1}^\infty \in l^2,
   \ \ x\in \R^n.
\ee
Then by 
  \eqref{sig1}-\eqref{sig4a}  we infer that
for every $v\in V_2 $ and $t\in \R$,
the operator 
  $\sigma (t, v): l^2\to H$ is a   Hilbert-Schmidt 
  operator 
   with norm
 $$
 \|  B (t,v)   \|^2_{\call_2 (l^2, H)}
 =
 \sum_{i=1}^\infty
 \|  \sigma_i (t,\cdot, v ) \|^2_H
  $$ 
   $$
   \le 2 
 \sum_{i=1}^\infty
  \int_{\o} \left (
   |  \sigma_{1,i}  (t,x)|^2 
     +      | \sigma_{2,i} ( v(x) ) 
       |^2 \right ) dx
  $$
 $$
  \le
  2 \sum_{i=1}^\infty 
  \|\sigma_{1, i} (t) \|^2_ {L^2(\R^n)}
  + 2 |\o|
  \sum_{i=1}^\infty \beta_i 
  +2\sum_{i=1}^\infty
  \gamma_i\int_{\o}  |v(x)|^q dx
 $$
  \be\label{sig4b}
  \le
  2 \sum_{i=1}^\infty 
  \|\sigma_{1, i} (t) \|^2_ {L^2(\R^n)}
  + 2 |\o|
  \sum_{i=1}^\infty \beta_i 
  +2|\o|^{\frac {p-q}p} \sum_{i=1}^\infty
  \gamma_i  \| v \|_{V_2} ^q dx,
\ee
 where $|\o|$ is the volume of $\o$.
 
 The next lemma is concerned with the
 continuity of $B(t,v)$ with respect to $v$ in $ V_2$.
 
 \begin{lem}
 Suppose \eqref{sig1}-\eqref{sig4} hold.
 
 (i)  If $v, v_n \in V_2$ such that
   $\{v_n\}_{n=1}^\infty$ is bounded in $V_2$ and 
    $v_n \to v$ in $H$
   as $n\to \infty$, then for 
  all $t\in [0,T]$,
   \be\label{rde1 1}
  \lim_{n\to \infty}
   B(t,   v_n) =  B(t,  v)
   \ \text{ in }  \ \call_2 (l^2,H).
\ee

(ii) 
If  $ v\in   L^{p}
(0,T; V_2)
  $
  and 
  $\{v_n\}_{n=1}^\infty
 $ is a bounded sequence in
 $    L^{p}
(0,T; V_2)
  $
  such that
  $v_n  \to v$ in $L^2(0,T; H)$,
  then
  \be\label{rde1 2}
   \lim_{n\to \infty}
   B(\cdot ,   v_n) =  B(\cdot,  v)
   \ \text{ in }  \  L^2(0,T; \call_2 (l^2,H)).
\ee
 \end{lem}
 
 \begin{proof}
 (i).   Suppose $v, v_n \in V_2$ such that
   $\{v_n\}_{n=1}^\infty$ is bounded in $V_2$ and 
    $v_n \to v$ in $H$. Then there exists  a constant $c_1>0$
    such that
    \be\label{rde1 p1}
    \int_\o | v_n (x)|^p dx \le c_1,
    \quad \forall \ n\in \N.
    \ee
    Since  $v_n \to v$ in $H$, we find that
    there exists a subsequence 
    $\{v_{n_k}\}$ of $\{v_n \}_{n=1}^\infty$
    such that
       \be\label{rde1 p2}
      \lim_{k\to \infty}
       v_{n_k} (x) =  v(x),
       \quad \text{for almost all } \ x\in \o.
     \ee
     Since $q<p$, by   \eqref{rde1 p1} we infer that
     the sequence $\{|v_{n_k}|^q\}_{k=1}^\infty$
     is uniformly integrable on $\o$, which along with \eqref{rde1 p2}
     implies that
         \be\label{rde1 p2a}
     \lim_{k\to \infty}
      \int_\o  |  v_{n_k} (x) |^q dx
      =\int_\o |v(x)|^q dx. 
     \ee
     
     On the other hand,
     by \eqref{rde1 p2} and the continuity of $\sigma_{2,i}$ we get,
     \be\label{rde1 p3}
      \sigma_{2,i} (v_{n_k} (x))  \to 
      \sigma_{2,i} (v(x)),
       \quad \text{for almost all } \ x\in \o.
     \ee
    By   \eqref{sig3} we have
  $$
        |  \sigma_{2,i} (v_{n_k} (x)) - \sigma_{2,i} (v(x))  |^2
$$
      \be\label{rde1 p4}
       \le
       2 |  \sigma_{2,i} (v_{n_k} (x)) |^2
       +2 |  \sigma_{2,i} (v(x))  |^2
          \le
         4 \beta_i
         +2 \gamma_i |v_{n_k} (x)|^q
         +2 \gamma_i |v (x)|^q.
       \ee
       Denote by
         \be\label{rde1 p5}
       \phi_k (x)
       =
          4 \beta_i
         +2 \gamma_i |v_{n_k} (x)|^q
         +2 \gamma_i |v (x)|^q,
         \ee
         and
           \be\label{rde1 p5a}
       \phi (x)
       =  4 \beta_i
         + 4 \gamma_i |v (x)|^q.
        \ee
         By \eqref{rde1 p4}-\eqref{rde1 p5} we get
            \be\label{rde1 p6}
   |  \sigma_{2,i} (v_{n_k} (x)) - \sigma_{2,i} (v(x))  |^2
         \le \phi_k (x).
         \ee
         By \eqref{rde1 p2}-\eqref{rde1 p2a}
         and \eqref{rde1 p5}  we have
         \be\label{rde1 p7} 
          \lim_{k\to \infty}
           \phi_k (x) =\phi (x),
          \quad \text{for almost all } \ x\in \o,
           \ee
           and
            \be\label{rde1 p8} 
          \lim_{k\to \infty}\int_\o
           \phi_k (x) dx  =\int_\o \phi (x) dx.
           \ee
           
           It follows from   \eqref{rde1 p3},
           \eqref{rde1 p6},  \eqref{rde1 p7}-\eqref{rde1 p8}
           and a generalized Lebesgue dominated
           convergence theorem
           that   for every $i\in \N$,
      \be\label{rde1 p9} 
          \lim_{k\to \infty}\int_\o
  |  \sigma_{2,i} (v_{n_k} (x)) - \sigma_{2,i} (v(x))  |^2
            dx
          =0.
           \ee
            Note that  \eqref{sig1} and  \eqref{sig4a} 
            imply  that
            for all $m\in \N$,
            $$
            \| B(t, v_{n_k} ) - B(t, v)\|^2_{\call_2
            (l^2, H)}
            =\sum_{i=1}^\infty \int_\o
             |  \sigma_{2,i} (v_{n_k} (x)) - \sigma_{2,i} (v(x))  |^2
            dx
            $$
            \be\label{rde1 p10}
             =\sum_{i=1}^m \int_\o
             |  \sigma_{2,i} (v_{n_k} (x)) - \sigma_{2,i} (v(x))  |^2
            dx
            +
            \sum_{i=m+1}^\infty \int_\o
             |  \sigma_{2,i} (v_{n_k} (x)) - \sigma_{2,i} (v(x))  |^2
            dx.
         \ee
         
         For the last term in \eqref{rde1 p10},
         by \eqref{rde1 p4} and \eqref{rde1 p1} we get
         $$
          \sum_{i=m+1}^\infty \int_\o
            |  \sigma_{2,i} (v_{n_k} (x)) - \sigma_{2,i} (v(x))  |^2
            dx
            $$
            $$
            \le
          4   \sum_{i=m+1}^\infty  \beta_i |\o|
       +2   \sum_{i=m+1}^\infty  \gamma_i 
         \int_\o |v_{n_k} (x)|^q dx
         +2   \sum_{i=m+1}^\infty  \gamma_i
         \int_\o  |v (x)|^q dx
         $$
            $$
            \le
           4   \sum_{i=m+1}^\infty  \beta_i |\o|
         +2   \sum_{i=m+1}^\infty  \gamma_i 
         \left ({\frac {2(p-q)}{p}} |\o|+
         \int_\o |v_{n_k} (x)|^p  dx 
         +
         \int_\o  |v (x)|^p dx \right )
         $$
            \be\label{rde1 p11}
            \le
           4   \sum_{i=m+1}^\infty  \beta_i |\o|
         +2   \sum_{i=m+1}^\infty  \gamma_i 
         \left ({\frac {2(p-q)}{p}} |\o|+
         c_1
         +
         \int_\o  |v (x)|^p dx \right ).
         \ee
    By  \eqref{sig4}
   and
   \eqref{rde1 p11}
 we find  that
   for every $\eps>0$, there exists
   $m_0 =m_0 (\eps) \in \N$ such that
   for all $k\in \N$,
    \be\label{rde1 p15}
          \sum_{i=m_0 +1}^\infty \int_\o
  |  \sigma_{2,i} (v_{n_k} (x)) - \sigma_{2,i} (v(x))  |^2
            dx
            <{\frac 12} \eps.
         \ee
         
         On the other hand,
         by \eqref{rde1 p9}
         we infer that
         there exists $K=K(\eps) \in \N$ such that
         for all $k\ge K$,
           \be\label{rde1 p16}
          \sum_{i=1}^{m_0}  \int_\o
            |  \sigma_{2,i} (v_{n_k} (x)) - \sigma_{2,i} (v(x))  |^2
            dx
            <{\frac 12} \eps.
         \ee
         Then by \eqref{rde1 p10}
         and \eqref{rde1 p15}-\eqref{rde1 p16}
         we get for  all $t\in [0,T]$,
           $$
          \lim_{k\to \infty}
            \| B(t, v_{n_k} ) - B(t, v)\|^2_{\call_2
            (l^2, H)} =0.
           $$
           By a contradiction argument, one can verify
           that  for all $t\in [0,T]$,
           $$
          \lim_{n\to \infty}
            \| B(t, v_{n} ) - B(t, v)\|^2_{\call_2
            (l^2, H)} =0,
           $$
         which proves (i).

         (ii).   
       Suppose  $ v\in   L^{p}
(0,T; V_2)
  $
  and 
  $\{v_n\}_{n=1}^\infty
 $ is a bounded sequence in
 $    L^{p}
(0,T; V_2)
  $
  such that
  $v_n  \to v$ in $L^2(0,T; H)$.
        Then there exists  a constant $c_2>0$
    such that
    \be\label{rde1 q1}
   \int_0^T  \int_\o | v_n (t, x)|^p dx dt  \le c_2,
    \quad \forall \ n\in \N.
    \ee
    Since  $v_n  \to v$ in $L^2(0,T; H)$,
     we find that
    there exists a subsequence 
    $\{v_{n_k}\}$ of $\{v_n \}_{n=1}^\infty$
    such that
       \be\label{rde1 q2}
      \lim_{k\to \infty}
       v_{n_k} (t, x) =  v(t, x),
       \quad \text{for almost all }  (t,x)\in [0,T]
       \times \o .
     \ee
     Since $q<p$, by   \eqref{rde1 q1} we 
     see that
     the sequence $\{|v_{n_k}|^q\}_{k=1}^\infty$
     is uniformly integrable on $ [0,T]
       \times \o$, 
       and thus, by   \eqref{rde1 q2}
   we obtain 
         \be\label{rde1 q2a}
     \lim_{k\to \infty}
     \int_0^T  \int_\o  |  v_{n_k} (t, x) |^q dx dt
      =\int_0^T \int_\o |v(t, x)|^q dx dt. 
     \ee
     
    By \eqref{rde1 q2} and the continuity of $\sigma_{2,i}$ we get
     \be\label{rde1 q3}
      \sigma_{2,i}  (v_{n_k} (t,x))  \to 
      \sigma_{2,i} (v(t,x)),
       \quad \text{for almost all } 
       (t,x)\in
       [0,T]
       \times \o.
     \ee
     By \eqref{rde1 q2a}-\eqref{rde1 q3}
     and the argument of \eqref{rde1 p9} we infer that
            for every $i\in \N$,
      \be\label{rde1 q9} 
          \lim_{k\to \infty}
          \int_0^T \int_\o
          |  \sigma_{2,i} (v_{n_k} (t, x)) - \sigma_{2,i} (v(t, x))  |^2
            dx dt
          =0.
           \ee
            Note that   \eqref{sig4a} implies that
            for all $m\in \N$,
            $$
            \| B(t, v_{n_k} ) - B(t, v)\|^2_{L^2(0,T;
            \call_2
            (l^2, H))}
            $$
     $$
             =\sum_{i=1}^m
             \int_0^T  \int_\o
              |  \sigma_{2,i} (v_{n_k} (t, x)) - \sigma_{2,i} (v(t, x))  |^2
            dx dt
            $$
                   \be\label{rde1 q10}
            +
            \sum_{i=m+1}^\infty
            \int_0^T \int_\o
             |  \sigma_{2,i} (v_{n_k} (t, x)) - \sigma_{2,i} (v(t, x))  |^2
            dx dt.
         \ee
        For the last term in \eqref{rde1 q10},
         by the argument of
         \eqref{rde1 p11}  we get
         $$
          \sum_{i=m+1}^\infty
          \int_0^T  \int_\o
              |  \sigma_{2,i} (v_{n_k} (t,x)) - \sigma_{2,i} (v(t,x))  |^2
            dxdt
            $$
      \be\label{rde1 q11}
            \le
          4   \sum_{i=m+1}^\infty  \beta_i |\o|T
      +2   \sum_{i=m+1}^\infty  \gamma_i 
         \left ({\frac {2(p-q)}{p}} |\o| T +
         c_2
         +
         \int_0^T \int_\o  |v (t, x)|^p dx dt \right ).
         \ee
    By  \eqref{sig4}
   and
   \eqref{rde1 q11} we infer that
   for every $\eps>0$, there exists
   $m_0 =m_0 (\eps) \in \N$ such that
   for all $k\in \N$,
    \be\label{rde1 q15}
          \sum_{i=m_0 +1}^\infty 
          \int_0^T \int_\o
             |  \sigma_{2,i} (v_{n_k} (t,x)) - \sigma_{2,i} (v(t,x))  |^2
            dx dt
            <{\frac 12} \eps.
         \ee
         By \eqref{rde1 q9}
         we see  that
         there exists $K=K(\eps) \in \N$ such that
         for all $k\ge K$,
           \be\label{rde1 q16}
          \sum_{i=1}^{m_0} 
          \int_0^T  \int_\o
                |  \sigma_{2,i} (v_{n_k} (t,x)) - \sigma_{2,i} (v(t,x))  |^2
            dx dt
            <{\frac 12} \eps.
         \ee
         Then by \eqref{rde1 q10}
         and \eqref{rde1 q15}-\eqref{rde1 q16}
         we obtain,
           $$
          \lim_{k\to \infty}
            \| B(t, v_{n_k} ) - B(t, v)\|^2_{L^2(0,T;
            \call_2
            (l^2, H))} =0,
           $$
           which along with 
         a contradiction argument gives (ii)
         and thus completes the proof.
 \end{proof}

 In order to apply Theorem \ref{main}
 to establish the existence of
 martingale solutions  to 
 \eqref{rde1}-\eqref{rde3},  
define an operator  $A_1: V_1 \to V_1^*$ 
 by
 $$
 (A_1v, \  u)_{(V_1^*, V_1)}
 =- ((-\Delta)^{\frac s2} v,
 \ (-\Delta)^{\frac s2} u)_{H},
 \quad \forall \ v, u\in V_1.
 $$
 Then  we have:
\be\label{fA1}
A_1:  V_1 \to V_1^*
\ \text{  is continuous}.
\ee
 Let  $A_2(t) $ be 
 the Nemytskii
 operator associated with $f(t, \cdot, \cdot)$
 as given by 
 $A_2(t, v) (x)  = f(t, x, v(x))$ for all
 $v\in V_2$ and $x\in \R^n$.
Then by   \eqref{f3}  we infer that for $t\in [0,T]$,
\be\label{fA2}
A_2(t, \cdot) :
V_2 \to V_2^* \ \text{ is continuous}.
\ee
 Similarly,
  let $A_3(t) $ be the Nemytskii
 operator associated with $h(t, \cdot, \cdot)$
 as given by 
 $A_3(t, v) (x)  = h(t, x, v(x))$ for all
 $v\in V_3 $ and $x\in \R^n$.
Then by   \eqref{h1}  we infer that for $t\in [0,T]$,
\be\label{hA3}
A_3(t, \cdot) :
V_3 \to V_3^* \ \text{ is continuous}.
\ee
 Let $A(t): V \to V^*$ be  the 
 operator given by
 $A(t)=A_1(t) +A_2(t) +A_3(t)$
 for $t\in [0,T]$.

 With above notation,
   the stochastic equation \eqref{rde1}
 can be reformulated  as: 
$$
  du (t)
   = A(t, u(t)) dt  
   +  B  (t, u(t))    {dW}.
$$
  
  We are now ready to present the main result of this section.
  
  \begin{thm}\label{main_rde}
 If  \eqref{f1}-\eqref{sig4} hold,
 then for every $u_0 \in H$,
 system \eqref{rde1}-\eqref{rde3}
 has at least one martingale solution
 which satisfies the uniform estimates
 as given by \eqref{main 1}.
  \end{thm}
  
  \begin{proof}
  We will show that  all assumptions 
  of Theorem \ref{main} are satisfied.
  Note that the  hemicontinuity of $A$
  follows from \eqref{fA1}-\eqref{hA3},
  and thus {\bf (H1)} is fulfilled.
  
  On the other hand,
  by \eqref{f1} and \eqref{h1} we get,
    for $ t\in [0,T]$ and $u, v\in V$,
    $$
    2( A(t,u) -A(t,v),
    u-v)_{(V^*, V)}
    \le\| \varphi_3(t)\|_{L^\infty(\R^n)}
    \| u-v \|_H^2,
    $$
    where $\varphi_3\in L^1(0,T;
    L^\infty(\R^n))$, and hence
    {\bf (H2)}$^\prime$ is satisfied.
    By \eqref{f2},
    \eqref{h1},
      \eqref{sig4b}
      and Young's inequality,
       we get for $t\in [0,T]$  and $v\in V$,
       $$
       2( A(t,v), v)_{(V^*,V)}
       + \| B(t, v)\|^2_{\call_2(l^2, H)}
       \le
       -2 \| v\|^2_{V_1}
       -\delta_1 \| v\|^p_{V_2}
       $$
       \be\label{main_rde p1}
       +2 \|\varphi_1 (t)\|_{L^1(\R^n)}
       +2
   \|\varphi_3 (t)\|_{L^\infty(\R^n)}\|v\|^2_H
   +2\sum_{i=1}^\infty
   \|\sigma_{1,i}
   (t) \|^2_{L^2(\R^n)} + c_1,
 \ee
  where $c_1>0$ is  a constant
  depending on $\{\beta_i\}_{i=1}^\infty$
  and $\{\gamma_i\}_{i=1}^\infty$.
  Let $\alpha_2= \min\{ 2, \delta_1\}$. By
  \eqref{main_rde p1} we get
    for $t\in [0,T]$  and $v\in V$,
       $$
       2( A(t,v), v)_{(V^*,V)}
       + \| B(t, v)\|^2_{\call_2(l^2, H)}
       \le -\alpha_2
       \left (
       \| v\|_{V_1}^2
       +\| v\|_{V_2}^p
       +\| v\|_{V_3}^2
       \right )
       $$
             \be\label{main_rde p2}
             + \left (2
   \|\varphi_3 (t)\|_{L^\infty(\R^n)} 
   +\alpha_2
   \right )\|v\|^2_H
   +
         2 \|\varphi_1 (t)\|_{L^1(\R^n)}
    +2\sum_{i=1}^\infty
   \|\sigma_{1,i}
   (t) \|^2_{L^2(\R^n)} + c_1.
 \ee
 Since  $\varphi_1\in L^1(0,T;
    L^1(\R^n))$ and
   $\varphi_3\in L^1(0,T;
    L^\infty(\R^n))$,
   by \eqref{sig2}
   and \eqref{main_rde p2} we
   obtain
   {\bf (H3)}.
   
   Note that for $t\in [0,T]$
   and $v\in V$,
   \be\label{main_rde p3}
   \| A_1 (t,v)\|^2_{V_1^*}
   \le \| v \|^2_{V_1}.
   \ee
   On the other hand, by \eqref{f3} we get
   for $t\in [0,T]$
   and $v\in V$,
      \be\label{main_rde p4}
   \| A_2 (t,v)\|^{\frac p{p-1}}_{V_2^*}
   \le (2\delta_2)^{\frac p{p-1}}
   \| v\|^p_{V_2}
   + 2^{\frac p{p-1}}
   \int_{\R^n}
   |\varphi_2 (t,x)|^{\frac p{p-1}} dx,
   \ee
   where $\varphi_2 \in L^{\frac p{p-1}}
   ([0,T] \times \R^n)$.
   By \eqref{h1} we  have
    for $t\in [0,T]$
   and $v\in V$,
     \be\label{main_rde p5}
   \| A_3 (t,v)\|^2_{V_3^*}
   \le  \| \varphi_3 (t)\|_{L^\infty (\R^n)}
   \| v\|^2_{V_3}.
   \ee
   Then {\bf (H4)} follows from
   \eqref{main_rde p3}-\eqref{main_rde p5}.
   Finally, by \eqref{sig4b}-\eqref{rde1 2} we  get
    {\bf (H5)}, and hence by
    Theorem \ref{main}, we conclude that
    system \eqref{rde1}-\eqref{rde3} has 
    at least one  martingale
    solution. This completes the proof.
   \end{proof}

  In order to obtain the uniqueness of solutions
  of \eqref{rde1}-\eqref{rde3}, we need to
  impose
  additional  restrictions on the nonlinear terms.
  In what follows, we further assume    that 
    $f$ satisfies:
  for all $t, u_1, u_2\in \R$ and $x\in \R^n$,
$$
  \left (
  f(t,x, u_1) -f(t,x, u_2)\right )
  (u_1-u_2)
  $$
     \be\label{ff1}
  \le -\delta_3 \left ( |u_1|^{p-2}
  + |u_2|^{p-2} \right ) (u_1 -u_2)^2
  + \varphi_4 (t,x) (u_1 -u_2)^2,
  \ee
  where $\delta_3>0$ is a constant
  and $\varphi_4
  \in L^1(0,T; L^\infty(\R^n))$.
  In addition,
  we also assume that for every $i\in \N$,
  $\sigma_{2,i}$
     satisfies:
  for all $u_1, u_2\in \R$,
  \be\label{ff2}
 | \sigma_{2,i} (u_1) -  \sigma_{2,i} (u_2)  |^2
 \le \alpha_i \left (
 1+ |u_1|^{q-2} + |u_2|^{q-2}
 \right ) |u_1-u_2|^2,
 \ee 
  where $\alpha_i>0$ is a constant
  such that $\sum_{i=1}^\infty \alpha_i <\infty$.
   
  The following result is concerned with
the uniqueness of solutions of \eqref{rde1}-\eqref{rde3}
under \eqref{ff1}-\eqref{ff2}.

  \begin{thm}\label{main_rde1}
 If  \eqref{f1}-\eqref{sig4} 
 and \eqref{ff1}-\eqref{ff2} hold,
 then the martingale solutions of 
  \eqref{rde1}-\eqref{rde3} are pathwise unique,
  and hence the system has a unique solution 
  in the sense of Definition \ref{dsol}.
  \end{thm}
  
  \begin{proof}
  We only need to check condition 
  {\bf (H2)}  in Theorem \ref{main}.
  By \eqref{sig4a}-\eqref{ff2} 
  and Young's inequality we get that
  for all 
 $u, v\in V_2$
and $t\in [0,T]$,
$$
 \|B(t,  u)-
B (t, v)\|^2_{\call_2(l^2, H) }
 \le   \sum_{i=1}^\infty \alpha_i 
 \int_{\o } 
   \big( 1+ |u(x)|^{q-2}  
+|v(x)|^{q-2}
\big)|u(x)-v(x)|^2dx
 $$
 \be\label{main-rde1 p1}
 \le
c_1
 \|u- v\|^2_H
 + \delta_3 
 \int_{\o} 
   \big(   |u(x)|^{p-2}  
+|v(x)|^{p-2}
\big)|u(x)-v(x)|^2dx,
\ee
 where
  $c_1>0$ is a constant depending   only
  on  $p, q, \delta_3$  and  $\sum_{i=1}^\infty
  {\alpha_i}$.
  By \eqref{h1}, \eqref{ff1} and
  \eqref{main-rde1 p1} we have
   for all 
 $u, v\in V_2$
and $t\in [0,T]$,
  $$
  2(A(t, u)-A(t,v),
  u-v)_{(V^*, V)}
  + \|B(t,  u)-
B (t, v)\|^2_{\call_2(l^2, H) }
$$
$$
\le
-2 \delta_3
  \int_{\o} 
   \big(   |u(x)|^{p-2}  
+|v(x)|^{p-2}
\big)|u(x)-v(x)|^2dx
+ 2 \int_\o
\left (\varphi_3 (t,x)
+\varphi_4 (t,x)
\right )|u(x)-v(x)|^2 dx
$$
$$
+ \|B(t,  u)-
B (t, v)\|^2_{\call_2(l^2, H) }
$$
$$
\le
\left (
c_1
+ \|  \varphi_3 (t)\|_{L^\infty (\R^n)}
+ \|  \varphi_4 (t)\|_{L^\infty (\R^n)}
\right ) \| u-v\|_H^2,
$$
which implies
{\bf (H2)} due to
the fact that
$\varphi_3, \varphi_4
\in L^1(0,T; L^\infty(\R^n))$.
Then the  pathwise uniqueness of solutions
follows from Theorem \ref{main},
that  completes the proof.
  \end{proof}

  \section{Appendix}

For reader's convenience,  we
  recall the following
Skorokhod-Jakubowski representation theorem
from \cite{brz1, jak2} on a topological space instead of
metric space.

\begin{prop}
\label{prop_sj}
 Suppose  $X$ 
 is  a topological space, and there
  exists a sequence
of continuous functions $g_n: X\rightarrow \mathbb{R}$ that separates points of $X$.
If    $\{\mu_n\}_{n=1}^\infty$
is a tight sequence of
   probability measures
  on $(X, \mathcal{B} (X) )$, then
  there exists a subsequence
  $\{\mu_{n_k}\}_{k=1}^\infty$
  of  $\{\mu_n\}_{n=1}^\infty$,
    a probability space $(
  \widetilde{
  \Omega},
  \widetilde{\mathcal{F}},
  \widetilde{\mathrm{P}})$, $X$-valued random variables
   $v_{ k}$  and $v$    such that
   the law of $v_{ k}$ is
   $\mu_{n_k}$ for all $k\in \mathbb{N}$ and
      $v_{ k} \rightarrow v$  $\widetilde{
  \mathrm{P}}$-almost surely
   as $k \rightarrow \infty$.
\end{prop}

\end{document}